\documentclass{daj}
\usepackage{mathtools}
\usepackage{amsthm}
\usepackage{amsfonts}
\usepackage{thmtools}
\usepackage{amssymb}
\usepackage{enumitem}

\declaretheorem[numberlike=equation]{conjecture}
\declaretheorem[name=Conjecture,numbered=no]{conjecture*}
\declaretheorem[numberlike=equation]{theorem}

\declaretheorem[name=Theorem,numbered=no]{theorem*}
\declaretheorem[numberlike=equation,style=defstyle]{definition}
\declaretheorem[unnumbered,name=Definition,style=defstyle]{definition*}

\declaretheorem[name=Lemma,numbered=no]{lemma*}

\declaretheorem[numberlike=equation]{corollary}
\declaretheorem[name=Corollary,numbered=no]{corollary*}

\declaretheorem[name=Proposition,numbered=no]{proposition*}

\declaretheorem[numberlike=equation]{claim}
\declaretheorem[name=Claim,numbered=no]{claim*}

\declaretheorem[numberlike=equation,style=defstyle]{remark}
\declaretheorem[unnumbered,name=Remark,style=defstyle]{remark*}

	\renewcommand{\vec}[1]{{\mathbf{#1}}}

	\makeatletter
	\newcommand{\va}{{\vec{a}}\@ifnextchar{^}{\!\:}{}}
	\newcommand{\vb}{{\vec{b}}\@ifnextchar{^}{\!\:}{}}
	\newcommand{\vc}{{\vec{c}}\@ifnextchar{^}{\!\:}{}}
	\newcommand{\vd}{{\vec{d}}\@ifnextchar{^}{\!\:}{}}
	\newcommand{\ve}{{\vec{e}}\@ifnextchar{^}{\!\:}{}}
	\newcommand{\vy}{{\vec{y}}\@ifnextchar{^}{\!\:}{}}
	\newcommand{\vs}{{\vec{s}}\@ifnextchar{^}{\!\:}{}}
	\newcommand{\vt}{{\vec{t}}\@ifnextchar{^}{\!\:}{}}
	\newcommand{\vx}{{\vec{x}}\@ifnextchar{^}{}{}}		
	\newcommand{\vz}{{\vec{z}}\@ifnextchar{^}{\!\:}{}}
	\newcommand{\vv}{{\vec{v}}\@ifnextchar{^}{\!\:}{}}
	\newcommand{\vu}{{\vec{u}}\@ifnextchar{^}{\!\:}{}}
	\newcommand{\vf}{{\vec{f}}\@ifnextchar{^}{\!\:}{}}
	\newcommand{\vg}{{\vec{g}}\@ifnextchar{^}{\!\:}{}}
	\newcommand{\vr}{{\vec{r}}\@ifnextchar{^}{\!\:}{}}
	\newcommand{\vw}{{\vec{w}}\@ifnextchar{^}{\!\:}{}}

	\newcommand{\vY}{{\vec{Y}}\@ifnextchar{^}{\!\:}{}}
	\newcommand{\vX}{{\vec{X}}\@ifnextchar{^}{}{}}		
	\newcommand{\vZ}{{\vec{Z}}\@ifnextchar{^}{\!\:}{}}
	\newcommand{\vG}{{\vec{G}}\@ifnextchar{^}{\!\:}{}}
	
	\makeatother

\newcommand{\C}{\mathbb{C}}
\newcommand{\F}{\mathbb{F}}

\newcommand{\cF}{{\mathcal{F}}}
\newcommand{\cI}{{\mathcal{I}}}
\newcommand{\cJ}{{\mathcal{J}}}
\newcommand{\cK}{{\mathcal{K}}}
\newcommand{\cL}{{\mathcal{L}}}
\newcommand{\cQ}{{\mathcal{Q}}}

\newcommand{\cT}{{\mathcal{T}}}

\newcommand{\spn}{\mathrm{span}}
\newcommand{\rank}{\mathrm{rank}}
\newcommand{\res}{\mathrm{Res}}
\newcommand{\eqdef}{\vcentcolon=}

\dajAUTHORdetails{%
  title = {Sylvester-Gallai type theorems for quadratic polynomials}, 
  author = {Amir Shpilka},
  plaintextauthor = {Amir Shpilka},
  keywords = {Sylvester-Gallai theorem, quadratic polynomials, polynomial identity testing},
}   

\dajEDITORdetails{%
   year={2020},
   number={13},
   received={26 June 2019},   
   revised={2 April 2020},    
   published={12 August 2020},  
   doi={10.19086/da.14492},       
}   

\begin{document}

\begin{frontmatter}[classification=text]

\title{Sylvester-Gallai type theorems for quadratic polynomials} 

\author[pgom]{Amir Shpilka
\thanks{The research leading to these results has received funding from the  Israel Science Foundation (grant number 552/16) and from the Len Blavatnik and the Blavatnik Family foundation}}

\begin{abstract}
We prove Sylvester-Gallai type theorems for quadratic polynomials. Specifically, we prove that if a finite collection $\cQ$, of irreducible polynomials of degree at most $2$, satisfy that for every two polynomials $Q_1,Q_2\in \cQ$ there is a third polynomial $Q_3\in\cQ$ so that whenever $Q_1$ and $Q_2$ vanish then also $Q_3$ vanishes, then the linear span of the polynomials in  $\cQ$ has dimension $O(1)$. We also prove a colored version of the theorem: If three finite sets of quadratic polynomials satisfy that for every two polynomials from distinct sets there is a polynomial in the third set satisfying the same vanishing condition then all polynomials are contained in an $O(1)$-dimensional space. 

This answers affirmatively two conjectures of Gupta \cite{Gupta14} that were raised in the context of solving certain depth-$4$ polynomial identities.

To obtain our main theorems we prove a new result classifying the possible ways that a quadratic polynomial $Q$ can vanish when two other quadratic polynomials vanish. Our proofs also require robust versions of a theorem of Edelstein and Kelly (that extends the Sylvester-Gallai theorem to colored sets).
\end{abstract}
\end{frontmatter}


\section{Introduction}\label{sec:intro}
The Sylvester-Gallai theorem asserts that if a finite set of points has the property that every line passing through any two points in the set also contains a third point in the set then all the points in the set are colinear. Many variants of this theorem were studied: extensions to higher dimensions, colored versions, robust versions and many more. For a survey on the Sylvester-Gallai theorem and its variants see \cite{BorwenMoser90}. One specific extension that is relevant to our work is the following colored version that was obtained by Edelstein and Kelly: If three finite sets of points satisfy that every line passing through points from two different sets also contains a point from the third set, then, in this case too all the points belong to a low dimensional space. 
 
\sloppy Another extension of the theorem that is relevant to our work was proved in \cite{barak2013fractional,DSW12}. There the authors proved the following robust version of the Sylvester-Gallai theorem (along with other robust versions of similar theorems): if a finite set of points satisfies that for every point $p$ in the set there is a $\delta$ fraction of other points so that for each of them, the line passing through it and $p$, spans a third point in the set, then the set is contained in an $O(1/\delta)$-dimensional space. 

While these theorems may seem unrelated to computation at first sight they have important consequences for locally decodable and locally correctable codes \cite{barak2013fractional,DSW12}, for reconstruction of certain depth-$3$ circuits \cite{DBLP:journals/siamcomp/Shpilka09,DBLP:conf/coco/KarninS09,DBLP:conf/coco/Sinha16} and for the polynomial identity testing (PIT for short) problem, which we describe next. 

The PIT problem asks to give a deterministic algorithm that given arithmetic circuit as input determines whether it computes the identically zero polynomial. This is a fundamental problem in theoretical computer science that has attracted a lot of attention both because of its intrinsic importance, its relation to other derandomization problems \cite{DBLP:journals/cc/KoppartySS15,Mulmuley-GCT-V,DBLP:conf/approx/ForbesS13,DBLP:journals/cacm/FennerGT19,DBLP:conf/stoc/GurjarT17,DBLP:conf/focs/SvenssonT17} and its connections to lower bounds for arithmetic circuits \cite{HeintzSchnorr80,DBLP:conf/fsttcs/Agrawal05,DBLP:journals/cc/KabanetsI04,DBLP:journals/siamcomp/DvirSY09,DBLP:journals/toc/ForbesSV18,DBLP:conf/coco/ChouKS18}. For more on the PIT problem see  \cite{SY10,Saxena09,Saxena14,ForbesThesis}. 

The case most relevant to Sylvester-Gallai type theorems is when the input circuit is a depth-$3$ circuit with small top fan-in. Specifically, a homogeneous $\Sigma^{[k]}\Pi^{[d]}\Sigma$ circuit in $n$ variables computes a polynomial of the following form
\begin{equation}\label{eq:sps}
\Phi(x_1,\ldots,x_n) = \sum_{i=1}^{k}\prod_{j=1}^{d}\ell_{i,j}(x_1,\ldots,x_n)\;,
\end{equation}
where each $\ell_{i,j}$ is a linear form.
Consider the PIT problem for $\Sigma^{[3]}\Pi^{[d]}\Sigma$ circuits. I.e., $\Phi$ is given as in \autoref{eq:sps} and it has $3$ multiplication gates, i.e. $k=3$. If $\Phi$ computes the zero polynomial then we have, for every $j,j'\in[d]$, that
$$\prod_{i=1}^{d}\ell_{1,i} \equiv 0 \mod \ell_{2,j},\ell_{3,j'}\;.$$
As the zero set of two linear functions is an irreducible variety, we get as a consequence that for every $j,j'\in[d]$, the linear functions $\ell_{2,j}$ and $\ell_{3,j'}$ span a linear function in $\{\ell_{1,1},\ldots,\ell_{1,d}\}$. In other words, the three sets $\cT_i = \{\ell_{i,1},\ldots,\ell_{i,d}\}$, for $i\in\{1,2,3\}$, satisfy the conditions of the Edelstein-Kelly theorem described above,\footnote{The theorem speaks about line through points rather than span of vectors, but it is not hard to see how to translate the Edelstein-Kelly theorem to this setting as well. See \autoref{rem:span-pass}.} and hence span a low dimensional space. Thus, if $\Phi\equiv 0$ then we can rewrite the expression for $\Phi$ using only constantly many variables (after a suitable invertible linear transformation). This allows efficient PIT algorithms for such $\Sigma^{[3]}\Pi^{[d]}\Sigma$ circuits. The case of more than $3$ multiplication gates is more complicated and satisfies a similar higher dimensional condition. This rank-bound approach for PIT of $\Sigma\Pi\Sigma$ circuits was raised in \cite{DBLP:journals/siamcomp/DvirS07} and later carried out in \cite{DBLP:conf/focs/KayalS09,DBLP:journals/jacm/SaxenaS13}.\footnote{The best algorithm for PIT of $\Sigma^{[k]}\Pi^{[d]}\Sigma$  circuits was obtained through a different, yet related, approach in \cite{DBLP:journals/siamcomp/SaxenaS12}.}

While such rank-bounds found important applications in studying PIT of depth-$3$ circuits, it seemed that such an approach cannot work for depth-$4$ $\Sigma\Pi\Sigma\Pi$ circuits,\footnote{Though we note that for multilinear $\Sigma\Pi\Sigma\Pi$ circuits Saraf and Volkovich obtained an analogous bound on the sparsity of the polynomials computed by the multiplication gates in a zero circuit \cite{DBLP:journals/combinatorica/SarafV18}.} even in the simplest case where there are only $3$ multiplication gates and the bottom fan-in is two, i.e., for homogeneous  $\Sigma^{[3]}\Pi^{[d]}\Sigma\Pi^{[2]}$ circuits that compute polynomials of the form
\begin{equation}\label{eq:spsp}
\Phi(x_1,\ldots,x_n) = \prod_{j=1}^{d}Q_{1,j}(x_1,\ldots,x_n)+\prod_{j=1}^{d}Q_{2,j}(x_1,\ldots,x_n)+\prod_{j=1}^{d}Q_{3,j}(x_1,\ldots,x_n)\;,
\end{equation}
where each $Q_{i.j}$ is a homogeneous quadratic polynomial. Indeed, we if try to reason as before then we get 
\begin{eqnarray}\label{eq:quad-pit}
\prod_{j=1}^{d}Q_{1,j}(x_1,\ldots,x_n) = 0 \mod Q_{2,j},Q_{3,j'}.
\end{eqnarray}
However, unlike the linear case it is not clear what  can be concluded now. Indeed, if a product of linear functions vanishes modulo two linear functions, then we know that one function in the product must be in the linear span of those two linear functions. For quadratic polynomials this is not necessarily the case. For example, note that if for a quadratic $Q$ we have that $Q=0$ and $Q + x^2=0$ then also $Q+xy=0$, and, clearly, we can find $Q$ such that $Q+xy$ is not spanned by $Q$ and $Q+x^2$. An even more problematic difference is that it may be the case that \autoref{eq:quad-pit} holds but that no $Q_{1,j}$ always vanishes when, say, $Q_{2,1},Q_{3,1}$ vanish. For example, let 
$$Q_1 = xy+zw \quad,\quad Q_2 = xy-zw \quad,\quad Q_3 = xw \quad,\quad Q_4 = yz.$$
Then, it is not hard to verify that 
$$Q_3 \cdot Q_4 \equiv 0 \mod Q_1,Q_2.$$
but neither $Q_3$ nor $Q_4$ vanish identically modulo $Q_1,Q_2$.
Thus, the PIT problem for sums of products of quadratics seem much harder than the corresponding problem for depth-$3$ circuits. Indeed, currently no efficient deterministic PIT algorithm is known for $\Sigma^{[3]}\Pi^{[d]}\Sigma\Pi^{[2]}$ circuits.

In spite of the above, Beecken et al.  \cite{DBLP:journals/iandc/BeeckenMS13,Gupta14} and Gupta \cite{Gupta14} conjectured that perhaps the difference between the quadratic case and the linear case is not so dramatic. In fact, they suggested that this may be the case for any constant degree and not just for quadratics. Specifically, Gupta observed that whenever \autoref{eq:quad-pit} holds it must be the case that there are four
polynomials in $\{Q_{1,j}\}$ whose product vanishes identically. That is, for every $(j,j')\in[d]^2$ there are $i_{1,j,j'}, i_{2,j,j'},i_{3,j,j'}, i_{4,j,j'}\in [d]$ so that 
$$Q_{1,i_{1,j,j'}}\cdot Q_{1,i_{2,j,j'}} \cdot Q_{1,i_{3,j,j'}}\cdot Q_{1,i_{4,j,j'}} \equiv 0 \mod Q_{2,j},Q_{3,j'}.$$
Gupta then raised the conjecture that whenever this holds for every $j,j'$ and for every two of the multiplication gates, then it must be the case that the \emph{algebraic} rank of the set $\{Q_{i,j}\}$ is $O(1)$. More generally, Gupta conjectured that this is the case for any fixed number of sets.

\begin{conjecture}[Conjecture 1 in \cite{Gupta14}]\label{con:gupta-general}
Let $\cF_1,\ldots, \cF_k$ be finite sets of irreducible homogenous polynomials in $\C[x_1,\ldots, x_n]$ of degree $\leq r$ such that $\cap_i \cF_i = \emptyset$ and for every $k-1$ polynomials $Q_1,\ldots,Q_{k-1}$, each from a distinct set, there are $P_1,\ldots,P_c$ in the remaining set such that  whenever $Q_1,\ldots,Q_{k-1}$ vanish then also the product $\prod_{i=1}^{c}P_i$ vanishes. Then, $\text{trdeg}_\C(\cup_i \cF_i) \leq \lambda(k, r, c)$ for some function $\lambda$, where trdeg stands for the transcendental degree (which is the same as algebraic rank).
\end{conjecture}

The condition in the conjecture can be stated equivalently as $$\prod_{i=1}^{c}P_i \in \sqrt{(Q_1,\ldots,Q_{k-1})},$$ where the object on the right hand side is the radical of the ideal generated by $\{Q_i\}_{i=1}^{k-1}$ (see \autoref{sec:res}). Note that for $r=1$ we have also $c=1$ and by the Edelstein-Kelly theorem $\lambda$ is $\leq 2$ in this case (and we can replace algebraic rank with linear rank).

\sloppy In \cite{DBLP:journals/iandc/BeeckenMS13} Beecken et al. conjectured that  the algebraic rank of simple and minimal $\Sigma^{[k]}\Pi^{[d]}\Sigma\Pi^{[r]}$ circuits (see their paper for definition of simple and minimal) is $O_k(\log d)$. We note that this conjecture is weaker than Gupta's as every zero $\Sigma^{[k]}\Pi^{[d]}\Sigma\Pi^{[r]}$ circuit gives rise to a structure satisfying the conditions of Gupta's conjecture, but the other direction is not necessarily true.
Beecken et al. also showed how to obtain a deterministic PIT for $\Sigma^{[k]}\Pi^{[d]}\Sigma\Pi^{[r]}$ circuits assuming the correctness of their conjecture. 



As an approach towards solving \autoref{con:gupta-general} Gupta set up a collection of conjectures, each of which is a natural extension of a known Sylvester-Gallai type theorem for the case of higher degree polynomials. 
The first conjecture is a direct analog of the Sylvester-Gallai theorem where we replace the requirement that a line through two points contains a third with a more algebraic condition: that for every two polynomials there is a third one so that whenever the two polynomials vanish then also the third vanishes.

\begin{conjecture}[Conjecture 2 of \cite{Gupta14}]\label{con:gupta-sg}
Let $Q_1,\ldots,Q_m \in \C[x_1,\ldots, x_n]$ be irreducible and homogenous polynomials of degree $\leq r$ such that for every pair of distinct $Q_i,Q_j$ there is a distinct $Q_k$ so that  whenever $Q_i$ and $Q_j$ vanish then so does $Q_k$. Then $\text{trdeg}_\C(Q_1,\ldots, Q_m) \leq \lambda(r)$.
\end{conjecture}

Note that Sylvester-Gallai's theorem is equivalent to the special case $r=1$.
A more general conjecture in \cite{Gupta14} is that a similar phenomenon holds when the polynomials come from different sets.

\begin{conjecture}[Conjecture 30 of \cite{Gupta14}]\label{con:gupta-ek} 
Let $R, B, G$ be finite disjoint sets of irreducible homogenous polynomials in $\C[x_1,\ldots,x_n]$ of degree $\leq r$  such that   for every pair $Q_1,Q_2$ from distinct sets there is a $Q_3$
in the remaining set so that whenever $Q_1$ and $Q_2$ vanish then also $Q_3$ vanishes. Then $\text{trdeg}_\C(R \cup B \cup G) \leq \lambda(r)$.\footnote{Here and in Conjectures~\ref{con:gupta-general} and \ref{con:gupta-sg} we actually need to assume that the polynomials are pairwise linearly independent.}
\end{conjecture}

The case $r=1$ is the Edelstein-Kelly theorem. Both \autoref{con:gupta-sg} and \autoref{con:gupta-ek} were open, prior to this work, for any degree $r>1$.



\subsection{Our Results}

Our main results give affirmative answers to \autoref{con:gupta-sg} and \autoref{con:gupta-ek} for the case $r=2$.  This shows that a Sylvester-Gallai type phenomenon holds for degree $2$ and we believe this indicates that this might be the case for higher degrees as well. Specifically we prove the following two theorems. The first is an extension of the Sylvester-Gallai theorem to quadratic polynomials. It confirms \autoref{con:gupta-sg}  for the case $r=2$.

\begin{theorem}\label{thm:main-sg-intro}
Let $\{Q_i\}_{i\in [m]}$ be homogeneous quadratic polynomials over $\C$ such that each $Q_i$ is either irreducible or a square of a linear function. Assume further that for every $i\neq j$ there exists $k\not\in \{i,j\}$ such that whenever $Q_i$ and $Q_j$ vanish $Q_k$ vanishes as well. Then the linear span of the $Q_i$'s has dimension $O(1)$.
\end{theorem}

The second theorem is an extension of the theorem of Edelstein-Kelly to quadratic polynomials, which gives an affirmative answer to \autoref{con:gupta-ek} for the case $r=2$..

\begin{theorem}\label{thm:main-ek-intro}
Let $\cT_1,\cT_2$ and $\cT_3$ be finite sets of homogeneous quadratic polynomials over $\C$ satisfying the following properties:
\begin{itemize}
\item Each $Q\in\cup_i\cT_i$ is either irreducible or a square of a linear function.\footnote{We replace a linear function with its square to keep the sets homogeneous of degree $2$.}
\item No two polynomials are multiples of each other (i.e., every pair is linearly independent).
\item For every two polynomials $Q_1$ and $Q_2$ from distinct sets there is a polynomial $Q_3$ in the third set so that whenever $Q_1$ and $Q_2$ vanish then also $Q_3$ vanishes. 
\end{itemize}
Then the linear span of the polynomials in $\cup_i\cT_i$ has dimension $O(1)$.
\end{theorem}

Note that what we proved is even stronger than what was conjectured in Conjectures~\ref{con:gupta-sg} and \ref{con:gupta-ek}. There the conjecture is that there is an upper bound on the algebraic rank whereas our results give an upper bound on the linear rank (which of course trivially implies an upper bound on the algebraic rank).

From the perspective of PIT our results do not imply  \autoref{con:gupta-general}, even for the case of $k=3$ and $r=2$, yet we believe they are a significant step in the direction of resolving this conjecture and obtaining a PIT algorithm for $\Sigma^{[3]}\Pi^{[d]}\Sigma\Pi^{[2]}$ circuits.

An important tool in the proof of \autoref{thm:main-sg-intro} is a result of \cite{barak2013fractional,DSW12} that gives a robust version of the Sylvester-Gallai theorem (see \autoref{sec:robust-SG}). For the proof of \autoref{thm:main-ek-intro} we need the following relaxation of the Edelstein-Kelly theorem. Roughly, three finite sets form a $\delta$-EK configuration if for every point $p$ in one set a $\delta$ fraction of the points in a second set satisfy that the line connecting each of them to $p$  passes through a point in the third set.

\begin{theorem}\label{thm:EK-robust}
Let $0<\delta \leq 1$ be any constant. Let $\cT_1,\cT_2,\cT_3\subset\C^n$ be disjoint finite subsets that form a $\delta$-EK configuration.  Then $\dim(\spn\{\cup_i \cT_i\}) \leq O(1/\delta^3)$.
\end{theorem}

This theorem is similar in nature to the results proved in \cite{barak2013fractional,DSW12} (see \autoref{thm:bdwy}) but it does not seem to directly follow from them.

\subsection{Proof Idea}\label{sec:proof-idea}

The basic tool in proving Theorems~\ref{thm:main-sg-intro} and \ref{thm:main-ek-intro} is the following result that characterizes the different cases when a quadratic polynomial is in the radical of the ideal generated by two other quadratics, i.e., that is vanishes when the two quadratic polynomials vanish.

\begin{theorem}\label{thm:structure-intro}
Let $Q,Q_1,Q_2$ be such that whenever $Q_1$ and $Q_2$ vanish then also $Q$ vanishes. Then one of the following cases hold:
\begin{enumerate}
\item $Q$ is in the linear span of $Q_1,Q_2$  \label{case:span:int}
\item \label{case:rk1:int}
There exists a non trivial linear combination of the form $\alpha Q_1+\beta Q_2 = \ell^2$ where  $\ell$ is a linear function 
\item There exist two linear functions $\ell_1$ and $\ell_2$ such that when setting $\ell_1=\ell_2=0$ we get that $Q,Q_1$ and $Q_2$ vanish.  \label{case:2:int}
\end{enumerate}
\end{theorem}

The theorem guarantees that unless $Q$ is in the linear span of $Q_1$ and $Q_2$ then $Q_1$ and $Q_2$ must satisfy a very strong property, namely, they must span a square of a linear function or they have a very low rank (as quadratic polynomials). The proof of this theorem is based on analyzing the resultant of $Q_1$ and $Q_2$ with respect to some variable. We now explain how this theorem can be used to prove \autoref{thm:main-sg-intro}.

Consider a set of polynomials $\cT=\{Q_i\}$ satisfying the condition of \autoref{thm:main-sg-intro}. If for every $Q\in\cT$ for at least, say, $(1/100)\cdot |\cT|$ of the polynomials  $Q_i\in \cT$ there is a another polynomial in $\spn(Q,Q_i)$ then the claim follows by the robust version of the Sylvester-Gallai theorem proved in \cite{barak2013fractional,DSW12} (\autoref{thm:bdwy}). So let us assume this is not the case. And in fact, let us assume that there are two polynomials $Q_1,Q_2\in\cT$ for which this does not hold. This means that at least $0.98$ fraction of the polynomials in $\cT$ satisfy Case~\ref{case:rk1:int} or Case~\ref{case:2:int} of \autoref{thm:structure-intro} with $Q_1$ and $Q_2$. This gives very strong restriction on the structure of these $0.98\cdot|\cT|$ polynomials. 

To use this structure we first show that the polynomials satisfying Case~\ref{case:rk1:int}  of \autoref{thm:structure-intro} with both $Q_1$ and $Q_2$ also span a low dimensional space (\autoref{cla:2-2}). The intuition is that every such polynomial can be represented as both $\alpha Q_1 + \ell_1^2$ and as $\beta Q_2 + \ell_2^2$. This gives rise to many different equations involving $Q_1$ and $Q_2$. Analyzing those equations we show that all those $\ell_i$ span a low dimensional space.

The remaining polynomials must satisfy Case~\ref{case:2:int} of \autoref{thm:structure-intro} with either $Q_1$ or $Q_2$. We then show (\autoref{cla:3}) that, under the conditions of \autoref{thm:main-sg-intro}, all the polynomials that satisfy Case~\ref{case:2:int} of \autoref{thm:structure-intro} with, say, $Q_1$ span a low dimensional space. The intuition is that if we map the linear functions in some ``minimal'' representation of $Q_1$ to a new variable $z$, then all these polynomials will be mapped to quadratics of the form $z\cdot \ell_i$. We then show that these $\ell_i$'s satisfy the usual Sylvester-Gallai condition and hence get a bound on their span.

The proof outline of \autoref{thm:main-ek-intro}  involves more cases, but it is still similar in spirit and is based on studying the case where our three sets do not satisfy the robust version of the Edelstein-Kelly theorem (\autoref{thm:EK-robust}).

To prove \autoref{thm:EK-robust} we would like to reduce to the robust version of  the Sylvester-Gallai theorem proved in \cite{barak2013fractional,DSW12}. For example, if all our sets are of the same size then their union forms a  $\delta/3$-SG configuration (see \autoref{sec:robust-SG}) and we can conclude using the result of \cite{barak2013fractional,DSW12}. Thus, the main issue is what to do when the sets are of very different sizes. When the largest set has size  polynomial in the size of the smallest set then we prove that by sampling a random subset of appropriate size from the largest set and taking its union with the two other sets we again get a $\delta/6$-SG configuration. This implies that the second largest and smallest sets live in an $O(1)$-dimensional space and hence all the sets span an $O(1)$-dimensional space. The proof of the case where the largest set is much larger than the smaller set is different and is based on a completely different covering argument. 

\subsection{Organization}
The paper is organized as follows. \autoref{sec:prelim} contains basic facts regarding the resultant and some other basic tools and notation, including the robust version of the Sylvester-Gallai theorem of \cite{barak2013fractional,DSW12}. In \autoref{sec:robust-EK} we define the notion of a $\delta$-EK configuration and prove \autoref{thm:EK-robust}. \autoref{sec:structure} contains the proof of our structure theorem (\autoref{thm:structure-intro}). In \autoref{sec:quad-SG} we give the proof of \autoref{thm:main-sg-intro} and in \autoref{sec:quad-EK} we prove \autoref{thm:main-ek-intro}. Finally in \autoref{sec:discussion} we discuss further directions and open problems.

\section{Preliminaries}\label{sec:prelim}

In this section we explain our notation, give some basic facts from algebra that will be useful in our proofs and state a robust version of the Sylvester-Gallai theorem.

We will mostly use the following notation. Greek letters $\alpha, \beta,\ldots$ denote scalars from the field.  Uncapitalized letters $a,b,c,\ldots$  denote linear functions and $x,y,z$ denote variables (which are also linear functions). We denote $\vx=(x_1,\ldots,x_n)$. Capital letters such as $A,Q,F$  denote quadratic polynomials whereas $V,U,W$ denote linear spaces. Calligraphic letters $\cal I,J,F,Q,T$  denote sets. For a positive integer $n$ we denote $[n]=\{1,2,\ldots,n\}$.

We will also need on the following version of Chernoff bound. See e.g. Theorem 4.5 in \cite{MU05-book}.

\begin{theorem}[Chernoff bound]\label{thm:chernoff}
Suppose $X_1,\ldots, X_n$ are independent indicator random variables. Let $\mu = E[X_i]$ be the expectation of $X_i$. Then,
$$\Pr\left[\sum_{i=1}^{n}X_i < \frac{1}{2}n\mu\right] < \exp(-\frac{1}{8}n\mu).$$
\end{theorem}

\subsection{Facts from algebra}\label{sec:res}

A notation that will convenient to use is that of a radical ideal. In this work we only consider the ring of polynomials $\C[\vx]$. An ideal $I \subseteq \C[\vx]$ is an abelian subgroup that is closed under multiplication by ring elements. We will denote with $(Q_1,Q_2)$ the ideal generated by two polynomials $Q_1$ and $Q_2$. I.e. $(Q_1,Q_2) = Q_1 \cdot \C[\vx] + Q_2\cdot \C[\vx]$. The radical of an ideal $I$, denoted $\sqrt{I}$, is the set of all ring elements $f$ satisfying that for some natural number $m$, $f^m\in I$. Hilbert's Nullstellensatz implies that if a polynomial $Q$ vanishes whenever $Q_1$ and $Q_2$ vanish then $Q\in\sqrt{(Q_1,Q_2)}$ (see e.g. \cite{CLO}). We shall often use the notation $Q\in\sqrt{(Q_1,Q_2)}$ to denote this vanishing condition.

A tool that will play an important role in the proof of \autoref{thm:structure-intro} is the resultant of two polynomials. As we only consider quadratic polynomials in this paper we restrict our attention to resultants of such polynomials. Let $F,G\in\C[\vx]$ be quadratic polynomials. View $F,G$ as polynomials in $x_1$ over $\C(x_2,\ldots,x_n)$. I.e.
$$F = \alpha x_1^2 + ax_1 + F_0 \quad \text{ and } G = \beta x_1^2 + bx_1 + G_0 \;.$$
Then, the resultant of $F$ and $G$ with respect to $x_1$ is the determinant of their Sylvester matrix
$$
\res_{x_1}(F,G) \eqdef 
\left| \begin{bmatrix} F_0 & 0 & G_0 & 0 \\
a & F_0 & b & G_0 \\
\alpha & a & \beta & b\\
0 & \alpha & 0 &\beta
\end{bmatrix}\right| \;.
$$
A useful fact is that if the resultant of $F$ and $G$ vanishes then they share a common factor. 

\begin{theorem}[See e.g. Proposition 8 in $\mathsection 5$ of Chapter 3 in \cite{CLO}]\label{thm:res}
Given $F,G\in\F[x]$ of positive degree, the resultant $\res_x(F,G)$ is an integer polynomial in the coefficients of $F,G$. Furthermore, $F$ and $G$ have a common factor in $\F[x]$ if and only if $\res_x(F,G) = 0$.
\end{theorem}

Finally, we shall define the rank of a quadratic polynomial as follows.

\begin{definition}\label{def:rank-s}
For a quadratic polynomial we denote with $\rank_s(Q)$ the minimal $r$ such that there are $2r$  linear functions $\{\ell_i\}_{i=1}^{2r}$ satisfying $Q=\sum_{i=1}^r \ell_{2i}\cdot \ell_{2i-1}$. We call such a representation a \emph{minimal representation} of $Q$.
\end{definition}

This is a slightly different definition than the usual way one defines rank of quadratic forms, but it is more suitable for our needs. We note that a quadratic $Q$ is irreducible if and only if $\rank_s(Q)>1$. The next claim shows that a minimal representation is unique in the sense that the space spanned by the linear functions in it is unique. 

\begin{claim}\label{cla:irr-quad-span}
Let $Q$ be an irreducible quadratic polynomial with $\rank_s(Q)=r$. Let $Q=\sum_{i=1}^{r}a_{2i-1}\cdot a_{2i}$ and $Q = \sum_{i=1}^{r}b_{2i-1}\cdot b_{2i}$ be two different minimal representations of $Q$. Then $\spn\{a_i\} = \spn\{b_i\}$.
\end{claim}

\begin{proof}
Note that if the statement does not hold then, w.l.o.g., $a_1$ is not contained in the span of the $b_i$'s. This means that when setting $a_1=0$ the $b_i$'s are not affected on the one hand, thus $Q$ remains the same function of the $b_i$'s, and in particular $\rank_s(Q|_{a_1=0})=r$, but on the other hand $\rank_s(Q|_{a_1=0})=r-1$ (when considering its representation with the $a_i$'s), in contradiction.
\end{proof}

\subsection{Robust Sylvester-Gallai theorem}\label{sec:robust-SG}
We will need the following theorem of Dvir et al. \cite{DSW12} that improves on an earlier work of Barak et al. \cite{barak2013fractional}.

We say that the points $v_1,\ldots,v_m$ in  $\C^d$ form a $\delta$-SG configuration if for every $i \in [m]$ there exists at least $\delta m$ values of $j \in [m]$ such that the line through $v_i, v_j$ contains a third point in the set. 

\begin{theorem}[Theorem 1.9 of \cite{DSW12}]\label{thm:bdwy}\sloppy
If $v_1,\ldots,v_m\in \C^d$ is a $\delta$-SG configuration then $\dim(\spn\{v_1,...,v_m\}) \leq 12/\delta$.
\end{theorem}

An easy consequence of the theorem is the following.

\begin{corollary}\label{cor:bdwy}
Let $0<\delta < 1$.
Assume $v_0, v_1,\ldots,v_m\in \C^d$ are such that for every $i \in [m]$ there exists at least $\delta m$ values of $j \in [m]$ such that the line through $v_i, v_j$ contains a third point in the set (i.e. the condition holds for all the points except, possibly, $v_0$). Then $\dim{v_0,v_1,...,v_m} < 50/\delta$.
\end{corollary}

\begin{proof}
The only way that $v_0, v_1,\ldots,v_m$ fail to be a $\delta$-SG configuration is if $v_0$ does not satisfy the condition. By considering all pairs $(v_i,v_j)$ that lie on a line with $v_0$ we conclude that either $v_0, v_1,\ldots,v_m$ is a $\frac{\delta}{2}$-SG configuration or $v_1,\ldots,v_m$ is. 
In any case, by \autoref{thm:bdwy}, we get that  $\dim{v_1,...,v_m} \leq 48/\delta$ and the total dimension is at most $50/\delta$. 
\end{proof}

\begin{remark}\label{rem:span-pass}
In our application we will have that the span of two points contains a third point. This does not change the theorems much as by picking a random subspace $H$, of codimension $1$, and replacing each point $p$ with $H\cap \spn\{p\}$ we get that $p_3\in\spn\{p_1,p_2\}$ iff $H\cap \spn\{p_3\}$ is on the line passing through $H\cap \spn\{p_1\}$ and $H\cap \spn\{p_2\}$.
\end{remark}

\section{Robust Edelstein-Kelly theorems}\label{sec:robust-EK}

In this section we prove \autoref{thm:EK-robust} as well as some extensions of it, which give robust versions of the following theorem of Edelstein and Kelly \cite{EdelsteinKelly66}.

\begin{theorem}[Theorem 3 of \cite{EdelsteinKelly66}]\label{thm:EK}
Let $\cT_i$, for $i\in [3]$, be disjoint finite subsets of $\C^n$ such that for every $i\neq j$ and any two points $p_1\in \cT_i$ and $p_2 \in \cT_j$ there exists a point $p_3$ in the third set that is on the line passing through $p_1$ and $p_2$. Then, any such $\cT_i$ satisfy that $\dim(\spn\{\cup_i \cT_i\})\leq 3$.
\end{theorem}

We would be interested in the case where the requirement in the theorem holds with some positive probability.
We say that the sets $\cT_1,\cT_2,\cT_3\subset \C^n$ form a $\delta$-EK configuration if for every $i \in [3]$ and $p\in \cT_i$, for every $j\in[3]\setminus\{i\}$ at least $\delta$ fraction of the points $p_j\in\cT_j$ satisfy that  $p$ and $p_j$ span some point in the third set.\footnote{Note that here we use the notion of span rather than a line passing through points. However, as noted in \autoref{rem:span-pass}, this does not make any real difference.} To ease the reading we state again \autoref{thm:EK-robust}.


\begin{theorem*}[\autoref{thm:EK-robust}]
Let $0<\delta \leq 1$ be any constant. Let $\cT_1,\cT_2,\cT_3\subset\C^n$ be disjoint finite subsets that form a $\delta$-EK configuration.  Then $\dim(\spn\{\cup_i \cT_i\}) \leq O(1/\delta^3)$.
\end{theorem*}


\begin{proof}[Proof of \autoref{thm:EK-robust}]
Denote $|\cT_i|=m_i$. Assume w.l.o.g. that $|\cT_1| \geq   |\cT_2| \geq |\cT_3|$. The proof distinguishes two cases. The first is when  $|\cT_3|$ is not too small and the second case is when it is much smaller than the largest set. 

\begin{enumerate}
\item {\bf Case $m_3 > m_1^{1/3}$: } \hfill

Let $\cT'_1\subset \cT_1$ be a random subset, where each element is samples with probability $m_2/m_1 = |\cT_2|/|\cT_1$. By the Chernoff bound (\autoref{thm:chernoff}) we get that, w.h.p., the size of the set is at most, say, $2m_2$. Further, the Chernoff bound also implies that for every $p\in \cT_2$ there are at least $(\delta/2)\cdot m_2$ points in $\cT'_1$ that together with $p$ span a point in $\cT_3$. Similarly, for every $p\in \cT_3$ there are at least $(\delta/2)\cdot m_2$ points in $\cT'_1$ that together with $p$ span a point in $\cT_2$. Clearly, we also have that for every point $p\in\cT'_1$ there are $\delta m_2$ points in $\cT_2$ that together with $p$ span a point in $\cT_3$. Thus, the set $\cT'_1\cup \cT_2\cup \cT_3$ is a $(\delta/8)$-SG configuration and hence has dimension $O(1/\delta)$ by \autoref{thm:bdwy}. 

Let $V$ be a subspace of dimension $O(1/\delta)$ containing all these points. Note that in particular, $\cT_2,\cT_3\subset V$. As every point $p\in \cT_1$ is a linear combination of points in $\cT_2\cup\cT_3$ it follows that the whole set has dimension $O(1/\delta)$.
\end{enumerate}

%

\begin{enumerate}
\item {\bf Case $  m_3  \leq m_1^{1/3}$:}\hfill

In this case we may not be able to use the sampling approach from earlier as $m_2$ can be too small and the Chernoff argument from above will not hold. 

We say that a point $p_1\in \cT_1$ is a neighbor of a point $p\in \cT_2\cup \cT_3$ if the space spaned by $p$ and $p_1$ intersects the third set. Denote with  $\Gamma_1(p)$ the neighborhood of  a point $p\in\cT_2\cup\cT_3$ in $\cT_1$.

\sloppy
Every two points $p\in\cT_2$  and $q\in\cT_3$ define a two-dimensional space that we denote  $V(p,q)=\spn\{p,q\}$. 

Fix $p\in \cT_2$ and consider those spaces $V(p,q)$ that contain points from $\cT_1$. Clearly there are at most $|\cT_3|$ such spaces. Any two different subspaces $V(p,q_1)$ and $V(p,q_2)$ have intersection of dimension $1$ (it is $\spn\{p\}$) and by the assumption in the theorem the union $\cup_{q\in\cT_3}V(p,q)$ covers at least $\delta m_1$ points of $\cT_1$.  Indeed, $\delta m_1$ points  $q_1\in \cT_1$ span a point in $\cT_3$ together with $p$. As our points are pairwise independent, it is not hard to see that if $q_3 \in \spn\{p,q_1\}$ then $q_1 \in \spn\{p,q_3\}=V(p,q_3)$

For each  subspace $V(p,q)$ consider the set $V(p,q)_1 = V(p,q)  \cap \cT_1$. 

\begin{claim}\label{cla:T1-intersect}
Any two such spaces $V(p,q_1)$ and $V(p,q_2)$ satisfy that either $V(p,q_1)_1= V(p,q_2)_1$ or $V(p,q_1)_1\cap V(p,q_2)_1=\emptyset$. 
\end{claim}

\begin{proof}
If there was a point $p'\in V(p,q_1)_1\cap V(p,q_1)_1$ then both $V(p,q_1)$ and $V(p,q_2)$ would contain $p,p'$ and as $p$ and $p'$ are linearly independent (since they belong to  $\cT_i$'s they are not the same point) that would make $V(p,q_1)=V(p,q_2)$. In particular we get $V(p,q_1)_1= V(p,q_2)_1$.
\end{proof}

As conclusion we see that at most $O(1/\delta^2)$ different spaces $\{V(p,q)\}_q$ have intersection at least $\delta^2/100 \cdot m_1$ with $\cT_1$. Let $\cI$ contain $p$ and a point from each of the sets $\{V(p,q)_1\}$ that have size at least $\delta^2/100 \cdot m_1$. Clearly $|\cI| \leq O(1/\delta^2)$. We now repeat the following process. As long as $\cT_2 \not\subset \spn\{\cI\}$ we pick a point  $p'\in \cT_2 \setminus \spn\{\cI\}$. We add $p'$ to $\cI$ along with a point from each large set $V(p',q)_1$, i.e. subsets satisfying $|V(p',q)_1|\geq \delta^2/100\cdot m_1$, and repeat. 

We next show that this process must terminate after $O(1/\delta)$ steps and that at the end $|\cI| = O(1/\delta^3)$. To show that the process terminates quickly we prove that if $p_k\in \cT_2$ is the point that was picked at the $k$'th step then $|\Gamma_1(p_k)\setminus \cup_{i\in[k-1]} \Gamma_1(p_i)| \geq (\delta/2)m_1$. Thus, every step covers at least $\delta/2$ fraction of new points in $\cT_1$ and thus the process must end after at most $O(1/\delta)$ steps. 

\begin{claim}\label{cla:one-large-intersect}
Let $p_i\in\cT_2$, for $i\in [k-1]$ be the point chosen at the $i$th step. If  the intersection of $V(p_k,q)_1$ with $V(p_i,q')_1$, for any $q,q'\in \cT_3$, has size larger than $1$ then $V(p_k,q)= V(p_i,q')$ (and in particular, $V(p_k,q)_1= V(p_i,q')_1$) and $|V(p_k,q)_1| \leq \delta^2/100 \cdot m_1$.

Moreover, if there is another pair $(q'',q''')\in\cT_3^2$ satisfying  $|V(p_k,q'')_1\cap  V(p_i,q''')_1|> 1$ then it must be the case that $V(p_i,q')=V(p_i,q''')$.
\end{claim}

\begin{proof}
If the intersection of $V(p_k,q)_1$ with $V(p_i,q')_1$ has size at least $2$ then by an argument similar to the proof of \autoref{cla:T1-intersect} we would get that $V(p_k,q) = V(p_i,q')$. To see that in this case the size of $V(p_i,q')_1$ is not too large we note that by our process, if $|V(p_i,q')_1|\geq \delta^2/100 \cdot m_1$ then $\cI$ contains at least two points from $V(p_i,q')_1$. Hence, $p_k\in V(p_i,q')\subset \spn\{\cI\}$ in contradiction to the choice of $p_k$.

To prove the moreover part we note that in the case of large intersection, since $V(p_k,q) = V(p_i,q')$, we have that $p_k,p_i\in V(p_i,q')$. If there was another pair $(q'',q''')$ so that $|V(p_k,q'')_1 \cap V(p_i,q''')_1|>1$ then we would similarly get that $p_k,p_i\in V(p_i,q''')$. By pairwise linear independence of the points in our sets this implies that $V(p_i,q')=V(p_i,q''')$.
\end{proof}



\begin{corollary}\label{cor:neighbor-grow}
Let $i\in[k-1]$ then 
$$|\Gamma_1(p_k)\cap \Gamma_1(p_i)|\leq \delta^2/100 \cdot m_1 + m_3^2.$$ 
\end{corollary}

\begin{proof}
The proof follows immediately from \autoref{cla:one-large-intersect}. Indeed, the claim assures that there is at most one subspace $V(p_k,q)$ that has intersection of size larger than $1$ with any $V(p_i,q')_1$ (and that there is at most one such subspace  $V(p_i,q')$) and that whenever the intersection size is larger than $1$ it is upper bounded by $\delta^2/100 \cdot m_1$. As there are at most $m_3^2$ pairs $(q,q')\in\cT_3^2$ the claim follows.
\end{proof}


The corollary implies that
$$|\Gamma_1(p_k)\cap \left( \cup_{i\in[k-1]} \Gamma_1(p_i)\right) | \leq k((\delta^2/100)m_1 + m_3^2) < (\delta/2)\cdot m_1,$$ 
where the last inequality holds for, say, $k<10/\delta$.\footnote{It is here that we use the fact that we are in the case $  m_3  \leq m_1^{1/3}$.}
As $|\Gamma_1(p_k)|\geq \delta \cdot m_1$, for each $k$, it follows that after $k<10/\delta$ steps 
$$|\cup_{i\in[k]} \Gamma_1(p_i)| > k(\delta/2)m_1.$$
In particular, the process must end after at most $2/\delta$ steps. 

As each steps adds to $\cI$ at most $O(1/\delta^2)$ vectors, at the end we have that $|\cI| = O(1/\delta^3)$ and every $p\in\cT_2$ is in the span of $\cI$. 

Now that we have proved that $\cT_2$ has small dimension we conclude as follows. We find a maximal subset of $\cT_3$ whose neighborhood inside $\cT_1$ are disjoint. As each neighborhood has size at least $ \delta \cdot m_1$ it follows there the size of the subset is at most $O(1/\delta)$. We add those $O(1/\delta)$ points to $\cI$ and let $V=\spn\{\cI\}$. Clearly $\dim(V) = O(1/\delta^3)$.

\begin{claim}
$\cup_i \cT_i \subset V$.
\end{claim}

\begin{proof}
We first note that if $p\in \cT_1$ is in the neighborhood of some $p'\in\cI\cap \cT_3$ then $p\in V$. Indeed, the subspace spanned by $p'$ and $p$ intersects $\cT_2$. I.e. there is $q\in \cT_2$ that is equal to $\alpha p + \beta p'$, where from pairwise independence both $\alpha\neq0$ and $\beta\neq 0$. As both $p'\in V$ and $\cT_2\subset V$ we get that also $p\in V$.

We now have that the neighborhood of every $p\in \cT_3\setminus \cI$ intersects the neighborhood of some $p'\in\cI\cap \cT_3$. Thus, there is some point $q\in \cT_1$ that is in $V$ (by the argument above as it is a neighbor of $p'$) and is also a neighbor of $p$. It follows that also $p\in V$ as the subspace spanned by $q$ and $p$ contains some point in $\cT_2$ and both $\{q\},\cT_2\subset V$ (and we use pairwise independence again). Hence all the points in $\cT_3$ are in $V$. As $\cT_2\cup\cT_3 \subset V$ it follows that also $\cT_1\subset V$.
\end{proof}
This concludes the proof of the case $ m_3  \leq m_1^{1/3}$.
\end{enumerate}
\end{proof}

\begin{remark}
The bound $O(1/\delta^3)$ is probably not tight and we believe that the correct bound should be $O(1/\delta)$ but we did not try to get tight bounds here. The theorem also seems similar in spirit to the results in  \cite{barak2013fractional,DSW12} but as far as we can tell it is not a direct corollary of any of the results there.
\end{remark}

\begin{remark}
While \autoref{thm:EK-robust} speaks about lines through points, a similar conclusion holds when we replace the condition that $p_3$ lies on the line through $p_1$ and $p_2$ with the condition $p_3\in\spn\{p_1,p_2\}$.
\end{remark}
%

Similar to \autoref{cor:bdwy} we have the following variant of \autoref{thm:EK-robust}.
\begin{theorem}\label{cor:EK-robust}
Let $0<\delta \leq 1$ be any constant. Let $\cT_1,\cT_2,\cT_3\subset\C^n$ be disjoint finite subsets. Assume that with the exception of at most $c$ elements from $\cup_{i=1}^{3}\cT_i$ all other elements in  $\cup_{i=1}^{3}\cT_i$  satisfy the $\delta$-EK property. Then $\dim(\spn\{\cup_i \cT_i\}) \leq O_c(1/\delta^3)$.
\end{theorem}

\begin{proof}[Sketch]
The proof is similar to the proof of \autoref{thm:EK-robust} so we just explain how to modify it.
\begin{enumerate}
\item {\bf Case  $m_3>m_1^{1/3}$:} Here too we repeat the sampling argument and note, similar to \autoref{cor:bdwy} that the sampled set give rise to an $\Omega(\delta/2^{c})$-SG configuration. Adding the $c$ '``bad'' elements to the subspace $V$ gives a subspace of dimension $O_c(1/\delta)$ spanning $\cT_2\cup\cT_3$. The rest of the proof is the same.

\item {\bf Case  $m_3\geq m_1^{1/3}$:} We repeat the covering argument only now we initiate $\cI$ with the $c$ '``bad'' elements. It is not hard to see that the rest of the proof gives the desired result.
\end{enumerate}
\end{proof}

For the proof of \autoref{thm:main-ek-intro} we would actually need the following extension of the theorem. The extension speaks of a situation where some linear combinations fall into a subspace $W$ and not just to one of the sets.

\begin{theorem}\label{thm:EK-robust-span}
Let $0<\delta \leq 1$ be any constant. 
Let $W\subset\C^n$ be an $r$-dimensional space and let $W_i\subset W$, for $i\in[3]$, be finite subsets of $W$. Let $\cK_1,\cK_2,\cK_3\subset\C^n\setminus W$ be finite subsets. Let $\cT_i=\cK_i\cup W_i$. 
Assume that no two vectors in $\cup_i \cT_i$ are linearly dependent. 

Assume that with the exception of at most $c$ elements from $\cup_{i=1}^{3}\cK_i$ all other elements satisfy the following relaxed EK-property: If $p\in \cK_i$ is not one of the $c$ exceptional points then for every $j\in[3]\setminus \{i\}$, for at least $\delta$ fraction of the points $q\in \cT_j$ the span of $p$ and $q$ contains a point in $\cT_k$, for the third index $k$. 
Then, there exists a linear subspace $V$ of dimension $\dim(V)= O_c(1/\delta^3)$ such that $\spn\{\cup_i \cT_i\} \subseteq W + V $. In particular, 
$\dim(\spn\{\cup_i \cT_i\}) \leq O_c(r+ 1/\delta^3)$.
\end{theorem}

Note that the theorem assumes nothing about the relation between the size of $W_i$ and $\cK_i$. Furthermore, it only asks that points in $\cK_i$ satisfy the spanning property with points from $\cT_j = \cK_j\cup W_j$ and that the spanned point can belong to $\cT_k = \cK_k\cup W_k$ and not just to $\cK_k$.

\begin{proof}
As in the proof of  \autoref{thm:EK-robust} we denote the neighborhood of an element $p\in \cup_i\cK_i$ in $\cT_j$ with $\Gamma_j(p)$. I.e., $q\in \cT_j$ belongs to $\Gamma_j(p)$, for $p\in \cK_i$ ($i\neq j$), if $p$ and $q$ span a point in $\cT_k$ where $k\not\in\{i,j\}$.
Assume w.l.o.g. that $|\cK_1|\geq |\cK_2|\geq |\cK_3|$.  As in the proof of \autoref{thm:EK-robust} we distinguish two cases.

\begin{enumerate}
\item {\bf Case $|\cK_3| > |\cK_1|^{1/3}$:} \hfill \break
Our first step is to project the space $W$ to a random one-dimensional space $W_0=\spn\{w_0\}\subset W$. We do so by projecting $\C^n$ to $\C^{n-r+1}$ in a way that the kernel of the projection is in $W$. Note that if we pick $w_0\in W$ at random (by, say, picking its coefficients uniformly from $[0,1]^r$) then any two vectors from $\cup_i\cK_i$ remain linearly independent with probability $1$.  We also note that this projection does not affect linear dependencies.

We abuse notation and use $\cK_i$ to denote the set $\cK_i$ after the projection.
In contrast to the $\cK_i$'s, all elements from $W$ now become linearly dependent. Thus, if $W_i$ was not empty then we now replace it with the single vector $w_0$.

We now proceed as in the proof of \autoref{thm:EK-robust} and sample a random subset $\cK'_1\subseteq \cK_1$ of size roughly $|\cK_2|$ (i.e. each element of $\cK_1$ is added to the set with probability $|\cK_2|/|\cK_1|$).  We would like to show that the new sets satisfy the conditions of \autoref{cor:EK-robust} with parameter $\delta/4$ and $c+1$ bad polynomials.

\begin{claim}\label{cla:ek-g-large}
$\cK'_1\cup\cK_2\cup\cK_3\cup\{w_0\}$ satisfy the conditions of \autoref{cor:EK-robust} with parameter $\delta/4$ and at most $c+1$ bad polynomials..
\end{claim}

\begin{proof}
Consider an element $p\in \cK'_1$ that is not exceptional. Then, before the projection to $W_0$, there were $\delta |\cT_2|$ elements of $\cT_2$ that each, together with $p$, spanned a point in $\cK_3 \cup W$. I.e., $|\Gamma_2(p)|\geq \delta|\cT_2|$. 
Observe that some of the elements from $\Gamma_2(p)$ may have been projected to a multiple of $w_0$.
We wish to show that in any case there are many points in $\cK'_1\cup\cK_2\cup\cK_3 \cup \{w_0\}$ that together with $p$ span a third point in the set. We consider two cases.
\begin{enumerate}
\item  $|\Gamma_2(p)\cap W| \leq (\delta/2)|\cT_2|$: In this case 
$$|\Gamma_2(p)\cap \cK_2| =  |\Gamma_2(p)| - |\Gamma_2(p)\cap W| \geq \delta |\cT_2| - (\delta/2)|\cT_2| = (\delta/2)|\cT_2| \geq (\delta/2)|\cK_2| \;.$$

\item $|\Gamma_2(p)\cap W| > (\delta/2)|\cT_2|$: Observe that for any $q\in \Gamma_2(p)\cap W$ the space spanned by $p$ and $q$ must contain a point in $\cK_3$ as otherwise we would get that $p\in W$ as well, contradicting the assumption that $\cK_1 \subset \C^n\setminus W$. Furthermore, all the points in $\cK_3$ that are obtained in this manner must be distinct. Indeed, if $p$ spans $q\in \cK_3$ with $w_1,w_2\in W_2$ then, as $w_1$ and $w_2$ are linearly independent and so are $p$ and $q$, we get that $\spn\{p,q\}=\spn\{w_1,w_2\}$ and again it follows that $p\in W$. It therefore follows that $p$ spans a point in $W$ with at least 
$$ (\delta/2)|\cT_2| \geq  (\delta/2)|\cK_2|$$ 
elements of $\cK_3$. As any two points in $\cup_i \cK_i$ remained linearly independent after the projection of $W$ to $W_0$, it follows that $p$ spans $w_0$ with at least $(\delta/2)|\cK_2|$ elements of $\cK_3$.
\end{enumerate}
Thus, in any case $p$ has at least  $(\delta/2) \cdot |\cK_2|$ points in $\cK_2\cup\cK_3$ that together with it span a third point in $\cK'_1\cup\cK_2\cup\cK_3\cup\{w_0\}$.

A similar argument shows roughly the same result for $p\in \cK_2\cup\cK_3$, where now we also have to remember to use the Chernoff bound to claim that the fraction of neighbors it has in $\cK'_1$ is roughly the same as in $\cK_1$, namely, at least $\delta/2$ (similarly to the proof of \autoref{thm:EK-robust}).

We therefore have a set of size at most, say, $4|\cK_2|+1$ that, with the possible exception of $c+1$ points (the original $c$ points and $w_0$), satisfy the condition of \autoref{cor:EK-robust} with parameter $\delta/4$ (we lose a factor of $2$ in $\delta$ when sampling $\cK'_1$ and then another factor due to the projection to $W_0$).
\end{proof}

We continue with the proof of \autoref{thm:EK-robust-span}. \autoref{cla:ek-g-large} and \autoref{cor:EK-robust} imply that the projected set $\cK'_1\cup\cK_2\cup\cK_3\cup W_0$ is contained in a subspace $V$ of dimension at most $O_c(1/\delta^3)$. As we projected $W$ to $\spn{w_0}\subseteq W$, it follows that $\cK'_1\cup\cK_2\cup\cK_3\cup W \subset V+W$, and clearly  $\dim(V+W)= O_c(r+1/\delta^3)$.
All that is left is to extend the bound to include $\cK_1$ instead of $\cK'_1$ and this is done as in the proof of \autoref{thm:EK-robust} without losing much in the dimension of $V$ (except a possible additive term of $c$ to $\dim(V)$). We thus get that $\cK_1\cup\cK_2\cup\cK_3\cup W\subseteq V+W $. This of course implies that $\cT_1\cup\cT_2\cup\cT_3\subseteq V+W $ and  $\dim(V+W)=O_c(r+1/\delta^3)$ as claimed.

\item {\bf Case $|\cK_3| \leq |\cK_1|^{1/3}$:} \hfill \break
The proof in this case is similar to the second case in the proof of \autoref{thm:EK-robust}. 

Note that, since $\cK_1$ is so large and every $p\in \cK_2$ has at least $\delta |\cT_1|$ neighbors in $\cT_1$, we get that  $p$ also has at least $(\delta/2)\cdot |\cK_1|$ neighbors in $\cK_1$. Indeed, as before if a neighbor $q$ of $p$ is in $W_1$ then the third point spanned by $p$ and $q$ cannot be in $W_3$. Hence it must be in $\cK_3$. Again it is easy to show that all the elements in $\cK_3$  that are obtained in this way must be distinct and since the set $\cK_3$  is too small the claim follows.

We now proceed as in  the proof of \autoref{thm:EK-robust}. 
For $p\in\cK_2$ and $q\in \cK_3$ we define the two dimensional space $V(p,q)=\spn\{p,q\}$ and denote $V(p,q)_1 = V(p,q)\cap \cK_1$.

Let $\cI$ contain the $c$ exceptional points. Consider $p_1\in\cK_2$ that is not in the span of the points in $\cI$. Add $p_1$ to $\cI$ as well as any $q\in \cK_3$ so that $|V(p_1,q)_1| > ((\delta/2)^2/100) \cdot|\cK_1|$. Continue this process where at each step $i$ we pick $p_i\in \cK_2$ that is not in the linear span of the vectors in $\cI$. We continue doing so noting that at east step the number of vectors in $\cK_1$ that is covered by the neighborhoods of the points $p_i$ that we picked grows by at least $(1/2)\cdot(\delta/2) \cdot |\cK_1|$ (the argument is the same as in the second case in the proof of  \autoref{thm:EK-robust}). Hence, the process must terminate after $O(1/\delta)$ steps at which stage $\cI$ is  of size $O_c(1/\delta^3)$. As in the second case in the proof of \autoref{thm:EK-robust} we conclude that $\cI$ spans all points in $\cK_2$. 

We continue as in the proof of \autoref{thm:EK-robust}. 
We find a maximal subset of $\cK_3$ whose neighborhood inside $\cT_1$ are disjoint. As each neighborhood has size at least $ \delta \cdot|\cT_1|$ it follows there the size of the subset is at most $O(1/\delta)$. We add those $O(1/\delta)$ points to $\cI$ and let $V=\spn\{\cI\}$. Clearly $\dim(V) = O_c(1/\delta^3)$. As in the  proof of \autoref{thm:EK-robust} we have that $\cT_2\subset V+W$.

\begin{claim}
$\cup_i \cT_i \subset V+W$.
\end{claim}

\begin{proof}
As before, note that if $p\in \cT_1$ is in the neighborhood of some $p'\in \cI\cap \cK_3$ then $p\in V+W$. Indeed, the subspace spanned by $p'$ and $p$ intersects $\cT_2$.  I.e. there is $q\in \cT_2$ that is equal to $\alpha p + \beta p'$, where from pairwise independence it follows that  $\alpha\neq0$ and $\beta\neq 0$. As $p'\in V$ and $\cT_2\subset V+W$ it also holds that $p\in V+W$.

We now have that the neighborhood of every $p\in \cK_3\setminus \cI$ intersects the neighborhood of some $p'\in\cI\cap \cK_3$. Thus, there is some point $q\in \cT_1$ that is in $V+W$ (by the argument above, as it is a neighbor of $p'$) and is also a neighbor of $p$. It follows that also $p\in V+W$ as the subspace spanned by $q$ and $p$ contains some point in $\cT_2$ and since $\cT_2\subset V+W$ we get that $p$ is in $V+W$ as well. Hence all the points in $\cK_3$ are in $V+W$. As $W_2\cup W_3\cup\cK_2\cup\cK_3 \subset V+W$ it follows that also $\cK_1\subset V+W$. The claim about the dimension of $V+W$ is clear.
\end{proof}
This concludes the proof of the second case and with it the proof of \autoref{thm:EK-robust-span}.
\end{enumerate}
\end{proof}

\section{Structure theorem for quadratics  satisfying $Q\in\sqrt{(Q_1,Q_2)}$}\label{sec:structure}


An important tool in the proofs of our main results is the following theorem that classifies all the possible cases in which a quadratic $Q$ is in the radical of two other quadratics, where all quadratics are irreducible.

Before stating the theorem we explain the intuition behind the different cases. We would like to understand when does a quadratic polynomial $Q$ can belong to the radical of two other quadratics. Clearly, if $Q$ is a linear combination of $Q_1,Q_2$ then it is in their radical (and in fact, in their linear span). Another option is that $Q_2 = \alpha Q_1 + b^2$ and then $Q$ can be of the form $\beta Q_1 + b\cdot a$. This case is clearly different than the linear span case. Finally, another option is the following situation: $Q'_1 = xy$, $Q'_2 = z(x+z)$ and $Q'=yz$. It is not hard to verify that in this case too, $Q'\in \sqrt{(Q'_1,Q'_2)}$. All these polynomials are reducible of course, but by defining, e.g., $Q_1 = Q'_1 + Q'_2$, $Q_2 = Q'_1-Q'_2$ and $Q= Q' + Q'_1 + Q'_2$ we get three irreducible polynomials that do not fall into any of the previous two cases.
Thus, all the three cases are distinct and can happen. What \autoref{thm:structure-intro} shows is that, essentially, these are the only possible cases. To ease the reading we repeat the theorem here with slightly different notation.

\begin{theorem}\label{thm:structure}
Let $Q,Q_1,Q_2$ be such that $Q \in \sqrt{Q_1,Q_2}$. Then one of the following cases hold:
\begin{enumerate}
\item $Q$ is in the linear span of $Q_1,Q_2$  \label{case:span}
\item \label{case:rk1}
There exists a non trivial linear combination of the form $\alpha Q_1+\beta Q_2 = b^2$ where  $b$ is a linear function 
\item There exist two linear functions $b_1$ and $b_2$ such that when setting $b_1=b_2=0$ we get that $Q,Q_1$ and $Q_2$ vanish. In other words, $Q,Q_1,Q_2\in \sqrt{(b_1,b_2)}$. \label{case:2}
\end{enumerate}
\end{theorem}

\begin{proof}
By applying a suitable linear transformation we can assume that for some $r\geq 1$  $$Q_1 = \sum_{i=1}^{r} x_i^2.$$

%

We can also assume wlog that $x_1^2$ appears only in $Q_1$ as we can replace $Q_2$ with any polynomial of the form $Q_2'=Q_2 - \alpha Q_1$ without affecting the result. Indeed, $Q\in \sqrt{(Q_1,Q_2)}$ if and only if $Q\in \sqrt{(Q_1,Q'_2)}$. Furthermore, all cases in the theorem remain the same if we replace $Q_2$ with $Q'_2$ and vice versa.

In a similar fashion we can replace $Q$ with $Q'=Q-\beta Q_1$ to get rid of the term $x_1^2$ in $Q$. Thus, wlog, the situation is 
\begin{eqnarray}
Q_1 &=& x_1^2 + Q_1' \nonumber \\
Q_2 &=& x_1\cdot b_2 -  A \label{eq:Q2}\\
Q  &=& x_1 \cdot b + B  \nonumber 
\end{eqnarray}
where $Q'_1,A,B,b_2$ and $b$ do not depend on $x_1$.

The first case we handle is when the ``new'' $Q_2$ does not depend on $x_1$.

\begin{claim}\label{cla:L2=0}
If $b_2=0$ then Case~\ref{case:rk1} of the theorem holds.
\end{claim}

\begin{proof}

For any assignment satisfying $A=0$ there are two solutions to $Q_1=0$, unless $Q'_1=0$, whereas $Q$ vanishes for only one value of $x_1$. Thus, we must have $Q'_1=0$ modulo $A$, which means that either $A$ is a square of a linear function and so $Q_1$ and $Q_2$ satisfy Case~\ref{case:rk1} of the theorem (as we assume $b_2=0$), or $Q_1 = \alpha \cdot A$ for some nonzero constant $\alpha$ and then $x_1^2$ is in the span of $Q_1$ and $Q_2$, and  again Case~\ref{case:rk1} of the theorem holds.
\end{proof}

%

We next handle the case where the ``new'' $Q_2$ is reducible.

\begin{claim}\label{cla:L-divides}
If $b_2$ divides $A$ then the conclusion of the theorem holds.
\end{claim}

\begin{proof}
If $b_2$ divides $A$ then $Q_2 = b_2 \cdot b'_2$. Assume that $b'_2$ is not a constant multiple of $b_2$ (as otherwise Case~\ref{case:rk1} of the theorem holds). Then, after a suitable invertible linear transformation we have $Q_2 = y\cdot z$.
Denote $$Q_1 = \alpha y^2 + \beta z^2 + y \cdot \ell_1 + z\cdot \ell_2 + Q''_1$$ and
$$Q = \alpha' y^2 + \beta' z^2 + \gamma' yz + y \cdot k_1 + z\cdot k_2 + Q'',$$
where $\ell_1,\ell_2,k_1,k_2,Q''_1,Q''$ do not involve $y$ nor $z$.
Observe that since we can subtract a multiple of $Q_2$ from $Q_1$ we can assume that the term $yz$ does not appear in $Q_1$.
Consider the assignment $y=0$. This simplifies $Q_1$ and $Q$ to:
$$Q_1|_{y=0} =  \beta z^2 + z\cdot \ell_2 + Q''_1$$ and
$$Q|_{y=0} = \beta' z^2 +z\cdot k_2 + Q'',$$
which are two polynomials not depending on $y$. We now have that any assignment that makes $Q_1|_{y=0}$ vanish, also makes  $Q|_{y=0}$ vanish. In other words $Q|_{y=0}\in \sqrt{(Q_1|_{y=0})}$. This means that all irreducible factors of $Q_1|_{y=0}$ divide $Q|_{y=0}$. Thus, either $Q|_{y=0} = \delta \cdot Q_1|_{y=0}$ for some constant $\delta$, or $Q_1|_{y=0}=b_3^2$ and $Q|_{y=0}=b_3\cdot b'_3$ for some linear functions $b_3,b'_3$. 

Notice that in the second case, if we set $y=b_3=0$ then $Q_1$ and $Q_2$ vanish and hence $Q$ also vanishes and Case~\ref{case:2} of the theorem holds.

So let us assume that $Q_1|_{y=0}$ divides $Q|_{y=0}$. We repeat the same reasoning when setting $z=0$ and again assume that $Q_1|_{z=0}$ divides $Q|_{z=0}$. By comparing coefficients we get that there are constants $\delta,\delta'$ such that $\beta'=\delta\beta, k_2 = \delta\ell_2,Q''=\delta Q''_1$ and $\alpha'=\delta'\alpha, k_1 = \delta'\ell_1,Q''=\delta' Q''_1$. It follows that either $\delta=\delta'$ and we obtain that $Q = \delta Q_1 + \gamma' Q_2$, which satisfies Case~\ref{case:span} of the theorem or that $Q''_1=Q_1=0$ in which case $Q_1,Q_2,Q$ all vanish when setting $y=z=0$ as in Case~\ref{case:2} of the theorem. 
\end{proof}

Hence, from now on we assume that $b_2$ is non-zero and does not divide $A$.
Consider the resultant of $Q_1,Q_2$ (as given in \autoref{eq:Q2}) with respect to $x_1$. It is equal to 
\begin{equation}\label{res}
\res_{x_1}(Q_1,Q_2) = A^2 + b_2^2 \cdot Q'_1.
\end{equation}
We next study what happens when the resultant vanishes. I.e. when
\begin{equation}\label{eq:res}
\res_{x_1}(Q_1,Q_2)=A^2 +  b_2^2 \cdot Q'_1=0 \;.
\end{equation}

\begin{claim}\label{cla:res-vanish}
Whenever $\res_{x_1}(Q_1,Q_2) =0$ it holds that $A\cdot b + b_2 \cdot B=0$. 
\end{claim}

\begin{proof}
If $\res_{x_1}(Q_1,Q_2)=0$ then either $b_2=0$, which also implies $A=0$ and in this case the claim clearly holds,  or  $b_2\neq 0$. Consider the case $b_2\neq 0$ and set $x_1=A/b_2$ (we are free to select a value for $x_1$ as $\res_{x_1}(Q_1,Q_2)$ does not involve $x_1$). Notice that for this substitution we have that $Q_2=0$ and that 
$$Q_1|_{x_1=A/b_2} =  (A/b_2)^2 +  Q'_1 = \res_{x_1}(Q_1,Q_2)/b_2^2=0.$$
Hence, we also have $Q|_{x_1=A/b_2} = 0$. In other words that 
$$A\cdot b + b_2 \cdot  B=0.$$
\end{proof}

In other words, \autoref{cla:res-vanish} implies  that 
$$A\cdot b + b_2 \cdot B \in \sqrt{(\res_{x_1}(Q_1,Q_2))}.$$
Thus, there exists an integer $k$ and a polynomial $\psi$ so that 
$$(A\cdot b + b_2 \cdot B)^k = \psi \cdot \res_{x_1}(Q_1,Q_2) = \psi \cdot (A^2  + b_2^2 \cdot Q'_1).$$
This means that all irreducible factors of $A^2 + b_2^2 \cdot Q'_1$ divide $A\cdot b + b_2 \cdot B$. As $\deg(A^2 +  b_2^2 \cdot Q'_1)=4$ and $\deg(A\cdot b + b_2 \cdot B)=3$ it follows, by examining the possible ways that a degree $4$ polynomial can factor, that  one of the following cases must hold: 
\begin{enumerate}
\item There is a quadratic polynomial $C$ and a linear function $a$ such that 
\begin{eqnarray*}
A^2 +   b_2^2 \cdot Q'_1=  C^2 \label{eq:deg-4-1}\\
b\cdot A +  b_2 \cdot B =   a \cdot C \label{eq:deg-3-1}
\end{eqnarray*} \label{case:quadratic-res}
\item For some scalar $\lambda$, a linear function $a$ and a quadratic $C$ \label{case:lin-res}
\begin{eqnarray}
A^2 +   b_2^2 \cdot Q'_1= a^2\cdot C \label{eq:deg-4}\\
b\cdot A +  b_2 \cdot B = \lambda \cdot a \cdot C \nonumber \label{eq:deg-3}
\end{eqnarray}
\end{enumerate}
We next handle each of these cases.
\begin{description}
\item[{Case~\ref{case:quadratic-res}:}] 
we have that 
$$b_2^2\cdot Q'_1 = C^2 - A^2 = (C+A)(C-A).$$
If $Q'_1$ is irreducible then $ \alpha  b_2^2 = (C+A)$ and $Q'_1 = \alpha (C-A)$, or $\alpha b_2^2 = (C-A)$ and $Q'_1 =  \alpha (C+A)$   for some $\alpha\neq 0$. In the first case we get that $Q'_1 = -2\alpha A + \alpha^2b_2^2$ and hence $Q_1  + 2\alpha Q_2 = (x+\alpha b_2)^2$. Similarly, in the second case we get $Q'_1 = 2\alpha A + \alpha^2b_2^2$ and thus $Q_1  - 2\alpha Q_2 = (x-\alpha b_2)^2$. In either cases, Case~\ref{case:rk1} of the theorem holds.

If $Q'_1$ is reducible, i.e. $Q'_1 = e\cdot f$, then either the analysis above continues to hold or it must be the case that (w.l.o.g.) $C+A = b_2 \cdot e$ and $C-A = b_2 \cdot f$. It follows that in this case $b_2$ divides $A$ and we are done by \autoref{cla:L-divides}.

\item[{Case~\ref{case:lin-res}:}]
From \autoref{eq:deg-4} we learn that $a^2 | \res_{x_1}(Q_1,Q_2)$ so in particular, when setting $a=0$ we get that the resultant is zero.  \autoref{thm:res} implies that, modulo $a$, either one of $Q_1,Q_2$ vanishes, or that   $Q_1$ and $Q_2$ share a linear factor. 

As $a$ does not involve $x_1$, clearly $Q_1|_{a=0}\neq 0$. Further, for $Q_2$ to vanish modulo $a$ we need that $b_2$ is a multiple of $a$, and vice versa. This implies that $b_2$ divides $A$ and we are done by \autoref{cla:L-divides}. 

We thus have to deal with the case that, modulo $a$,  $Q_1$ and $Q_2$ share a linear factor.
Let $a'$ be that common linear factor. We get that by setting $a=a'=0$ both $Q_1$ and $Q_2$ vanish and hence also $Q$ vanishes  and Case~\ref{case:2} of the theorem holds. 
\end{description}
This concludes the proof of \autoref{thm:structure}.
\end{proof}

\section{Sylvester-Gallai theorem for quadratic polynomials}\label{sec:quad-SG}

In this section we prove \autoref{thm:main-sg-intro}. For convenience we repeat the statement of the theorem.

\begin{theorem*}[\autoref{thm:main-sg-intro}]
Let $\{Q_i\}_{i\in [m]}$ be homogeneous quadratic polynomials such that each $Q_i$ is either irreducible or a  square of a linear function. Assume further that for every $i\neq j$ there exists $k\not\in \{i,j\}$ such that $Q_k\in\sqrt{(Q_i,Q_j)}$. Then the linear span of the $Q_i$'s has dimension $O(1)$.
\end{theorem*}


\begin{remark}
The requirement that the polynomials are homogeneous is not essential as homogenization does not affect the property $Q_k\in\sqrt{(Q_i,Q_j)}$.
\end{remark}

\subsection{Some useful claims}\label{sec:claims-sg}

In this section we look at some implications of \autoref{thm:structure}. We do so by considering two irreducible polynomials $Q_1$ and $Q_2$ and consider sets of polynomials that satisfy Case~\ref{case:rk1} or Case~\ref{case:2} of \autoref{thm:structure} with  $Q_1$ and $Q_2$.

\begin{claim}\label{cla:2-2}
Let $Q_1,Q_2$ be two linearly independent quadratic polynomials and let  $F_1,\ldots,F_m$ be quadratic polynomials such that for every $i$ there exist linear functions $\ell_i,b_i$ and a scalar $\beta_i$ so that 
\begin{equation}\label{eq:F-two-r}
F_i = Q_1 + \ell_i^2 = \beta_i\cdot Q_2 + b_i^2.
\end{equation}
Then, there exists a $4$-dimensional space $V$ such that for every $i$, $\{\ell_i,b_i\}\subseteq  V$.
\end{claim}

\begin{proof}
If $m\leq 2$ then the claim is trivial. We consider two cases.

\begin{enumerate}
\item {\bf Case 1: For all $i$, $\beta_i=\beta_1$:} \hfill \break
Let $V=\spn\{b_1,\ell_1\}$.
From the two representations of $F_1$ we get that
\begin{equation}\label{eq:2-2i}
Q_1 - \beta_1 Q_2 = b_1^2 - \ell_1^2 = (b_1-\ell_1)\cdot (b_1+\ell_1)\neq 0,
\end{equation}
where the fact that the expression above is nonzero follows as $Q_1$ and $Q_2$ are linearly indepdent.
Similarly, by considering the two representations of $F_j$ we get that
$$Q_1 - \beta_1 Q_2 = {b}_j^2 - {\ell}_j^2 = (b_j-\ell_j)\cdot (b_j+\ell_j).$$
Thus, $$(b_1-\ell_1)\cdot (b_1+\ell_1)=  (b_j-\ell_j)\cdot (b_j+\ell_j)\;.$$
Unique factorization implies that $b_j,\ell_j \in \spn\{b_i,\ell_1\} = V$ as claimed.

\item {\bf Case 2:  There is $j$ such that $\beta_j\neq \beta_1$:} \hfill\break
In this case we have that
$$F_1 = Q_1 + \ell_1^2 = \beta_1 Q_2 + {b}_1^2,$$
$$F_j = Q_1 + \ell_j^2 = \beta_j Q_2 + {b}_j^2,$$
and the matrix 
$$\begin{bmatrix} 1 & -\beta_1 \\
1 & -\beta_j
\end{bmatrix}
$$
is invertible. It follows that 
\begin{equation}\label{eq:Q1-span}
Q_2,Q_1 \in \spn\{(b_1-\ell_1)\cdot (b_1+\ell_1),(b_j-\ell_j)\cdot (b_j+\ell_j)\}.
\end{equation}
Let $$V = \spn\{b_1,b_j,\ell_1,\ell_j \}.$$ Consider any index $k$. W.l.o.g. $\beta_k\neq \beta_j$. Thus, as before, we get that 
$$Q_1 \in \spn\{(b_j-\ell_j)\cdot (b_j+\ell_j),(b_k-\ell_k)\cdot (b_k+\ell_k)\}.$$
Hence, either $Q_1 = \alpha \cdot (b_j-\ell_j)\cdot (b_j+\ell_j)$ or, for some $\alpha_3$ and a nonzero $\alpha_4$, 
\begin{equation}\label{eq:Q_1-2}
Q_1 =  \alpha_3   \cdot (b_j-\ell_j)\cdot (b_j+\ell_j) + \alpha_4 \cdot (b_k-\ell_k)\cdot (b_k+\ell_k).
\end{equation}
We first handle the later case. Combining \autoref{eq:Q_1-2} with \autoref{eq:Q1-span} we get that there exist constants $\alpha_1,\ldots,\alpha_4$, with $\alpha_4\neq 0$ so that
$$\alpha_1 \cdot (b_1-\ell_1)\cdot (b_1+\ell_1) + \alpha_2 \cdot (b_j-\ell_j)\cdot (b_j+\ell_j) = \alpha_3   \cdot (b_j-\ell_j)\cdot (b_j+\ell_j) + \alpha_4 \cdot (b_k-\ell_k)\cdot (b_k+\ell_k).$$
By switching sides it is easy to see that both $b_k-\ell_k$ and $b_k+\ell_k$ are spanned by the functions in $V$. In the former case where $Q_1 = \alpha \cdot (b_j-\ell_j)\cdot (b_j+\ell_j)$, it follows from \autoref{eq:F-two-r} (for $i=j$) that
$Q_2 = \frac{\alpha_1-1}{\beta_j} \cdot (b_j-\ell_j)\cdot (b_j+\ell_j)$. This contradicts the assumption that $Q_1,Q_2$ are linearly independent.
\end{enumerate}
\end{proof}

\begin{corollary}\label{cor:2-2}
Under the hypothesis of \autoref{cla:2-2}, there exist four linear functions $a_1,a_2,a_3,a_4$ such that every $F_i$ is a linear combination of $Q_1,\{a_i\cdot a_j\}_{i\leq j}$.
\end{corollary}

\begin{proof}\sloppy
Let $V$ be the subspace guaranteed by  \autoref{cla:2-2}. Let $\{a_1,\ldots,a_4\}$ be such that  $V=\spn\{a_1,a_2,a_3,a_4\}$. The claim follows immediate from the fact that each $\ell_i$ is a linear combination of $a_1,a_2,a_3,a_4$.
\end{proof}

\begin{claim}\label{cla:3}
Let $F_1,\ldots,F_{m'}$ be quadratics in our set\footnote{I.e. they are a subset of the $\{Q_i\}$ from the statement of \autoref{thm:main-sg-intro}.} that satisfy Case~\ref{case:2} of \autoref{thm:structure} with an irreducible $Q$. Then there exists an $O(1)$-dimensional space $V$ such that each $F_i$ is a quadratic polynomial in the linear functions in $V$.
\end{claim}

\begin{proof}
As $Q$ satisfies Case~\ref{case:2} of \autoref{thm:structure} and is irreducible it follows that $\rank_s(Q)= 2$ (recall \autoref{def:rank-s}). Thus, $Q$ is a quadratic polynomial in at most $4$ linear functions. Let $V$ to be the space spanned by the linear functions in a minimal representation of $Q$. By \autoref{cla:irr-quad-span} it follows that $V$ is well defined. Clearly $\dim(V)\leq 4$. 

Let $z$ be a new variable. Set each basis element of $V$ to a random multiple of $z$ (say by picking the multiples independently uniformly at random from $[0,1]$). Each $F_i$ now becomes $z\cdot b_i$ for some nonzero $b_i$. Indeed, if we further set $z=0$ then all linear functions in the representation of $Q$ vanish and hence also $F_i$ vanishes (this again follows from \autoref{cla:irr-quad-span}). Further, $b_i\neq 0$ as we mapped the basis elements to random multiples of $z$. We next show that unless all linear functions in the minimal representation of $F_i,F_j$ are in $V$ then $F_i,F_j$ remain linearly independent after this restriction.

\begin{claim}\label{cla:still-indep}
Let $V$ be a linear space of linear functions. Let $F = v_1 \cdot \ell_1 + v_2 \cdot \ell_2$ and  $G = v_3 \cdot \ell_3 + v_4 \cdot \ell_4$ be two linearly independent irreducible quadratics, where for every $i$, $v_i\in V$. If $\spn\{\ell_1,\ldots,\ell_4\} \not\subseteq V$ then with probability $1$, $F$ and $G$ remain linearly independent even after we map the basis elements of $V$ to random multiples of a new variable $z$ (say, by picking the multiples uniformly and independently from the segment $[0,1]$).
\end{claim}

We postpone the proof of \autoref{cla:still-indep} and continue with the proof of \autoref{cla:3}.
We next show that the linear functions $\{b_i\}_i\cup \{z\}$ satisfy the ``usual'' Sylvester-Gallai condition, i.e., that any two of them span a third function in the set (with the possible exception of $z$). In fact, we will add to this set all quadratics in our set that are now of the form $z\cdot \ell$ for a linear $\ell$. 

Consider two quadratics $Q_1=zb_1,Q_2=zb_2$ so that  neither $b_1$ nor $b_2$ is a multiple of $z$. If $\{b_1,b_2\}$ span $z$ then we are done. Otherwise, assume that $Q_3$ vanishes when $Q_1$ and $Q_2$ vanish. Then clearly $z$ divides $Q_3$. Thus $Q_3=zb_3$ and $b_3$ is in our set. Further, when we set $b_1=b_2=0$ both $Q_1$ and $Q_2$ vanish and hence $Q_3$ vanishes as well. Since $z\not \in \spn\{b_1,b_2\}$ this implies that $b_3 \in  \spn\{b_1,b_2\}$ and in this case too $b_1$ and $b_2$ span a third linear function in $\{b_i\}_i\cup \{z\}$. Note also that, by \autoref{cla:still-indep}, $b_3$ is not a multiple of $b_1$ nor of $b_2$ as this would imply that $Q_3$ and $Q_1$ (or $Q_2$) are linearly dependent in contradiction to our assumption. 

From \autoref{cor:bdwy} (recalling \autoref{rem:span-pass}) we get that the dimension of all those $\{b_i\}_i$ is $O(1)$. 

We now  repeat the same argument again for a different random mapping of the basis elements of $V$ to multiples of $z$. As before each $F_i$ is mapped to a polynomial of the form $z\cdot b'_i$ and again the dimension of $\{b'_i\}_i$ is $O(1)$. Let $U$ be the subspace containing the span of $V\cup \{b_i\}_i\cup \{b'_i\}_i$. Clearly $\dim(U)=O(1)$. We next show that every $F_i$ is a polynomial in the linear functions in $U$. Indeed, let $F = v_1 \cdot u_1 + v_2 \cdot u_2$ be arbitrary polynomial from $\{F_i\}_i$, where $v_1,v_2\in V$. Assume the first mapping mapped $v_i \mapsto \alpha_i \cdot z$ and the second mapping is $v_i \mapsto \beta_i \cdot z$. Then,  $F$ was mapped to $z\cdot b$ under the first mapping where $b = \alpha_1 u_1 + \alpha_2 u_2$ and to $z\cdot b'$ under the second mapping where $b' = \beta_1 u_1 + \beta_2 u_2$. As $\alpha_1,\alpha_2,\beta_1,\beta_2$ where chosen uniformly independently at random from $[0,1]$ it follows that 
the determinant $$\left|\begin{bmatrix} \alpha_1&  \alpha_2 \\
\beta_1 & \beta_2
\end{bmatrix}\right| \neq 0$$ and hence
$u_1,u_2 \in \spn\{b,b'\}\subseteq U$. As we also have $v_1,v_2\in V \subseteq U$ the claim follows.

This concludes the proof of \autoref{cla:3}.
\end{proof}

We now give the proof of  \autoref{cla:still-indep}.

\begin{proof}[Proof of  \autoref{cla:still-indep} ]
Let $x_1,\ldots,x_k$, for some $1\leq k\leq 4$ be a basis for    $\spn\{\ell_1,\ldots,\ell_4\} $ such that for some $0<t\leq k$, $x_{t+1},\ldots,x_k$ for a basis to $\spn\{\ell_1,\ldots,\ell_4\} \cap V$. We can rewrite $F$ and $G$ as 
$$F = \sum_{i=1}^{t}x_i u_i + F' \quad \text{ and } \quad G = \sum_{i=1}^{t}x_i w_i + G'$$
where $u_i,w_i\in V$ and $F',G'$ are defined over $V$, and, w.l.o.g., for every $i$, at least one of $u_i$ and $w_i$ is nonzero. 
Observe that $F$ and $G$ are linearly independent (over $\C$) if and only if the two vectors
$$u_F=(u_1,\ldots,u_t,F') \quad \text{ and } \quad w_G=(w_1,\ldots,w_t,G')$$
are linearly independent over $\C(V)$, the function field generated by adding the linear functions in $V$ to $\C$. 
Indeed, if $F$ and $G$ are linearly dependent over $\C$ then clearly $u_F$ and $w_G$ are linearly dependent over $\C$, and hence over $\C(V)$. If on the other hand $u_F$ and $w_G$ are linearly dependent over $\C(V)$ then this means that for some polynomials $f(V)$ and $g(V)$ we have
$$f \cdot (u_1,\ldots,u_t,F') = g  \cdot (w_1,\ldots,w_t,G').$$
This implies that all the $2\times 2$ determinants vanish, i.e. that $u_i\cdot w_j - u_j \cdot w_i=0$, for every $i$ and $j$, and $u_i \cdot G' - w_i \cdot F'=0$. 
By unique factorization we get that there are two possible cases:
\begin{enumerate}
\item There is $\alpha\in\C$ so that $u_i = \alpha w_i$ for some $i$: The equality $u_i\cdot w_j - u_j \cdot w_i=0$ implies that for every $j$ we actually have $u_j = \alpha w_j$, and the fact that $u_i \cdot G' - w_i \cdot F'=0$ gives $F'=\alpha G'$ and thus $u_F$ and $w_G$ are linearly dependent over $\C$ and indeed $F$ and $G$ are linearly dependent.
\item There are constants $\alpha_i$ such that for every $i$, $u_i= \alpha_i u_1$ and $w_i = \alpha_i w_1$: In this case, since $F$ is irreducible, it holds that $u_1$ does not divide $F'$. As $u_1 \cdot G' - w_1 \cdot F'=0$, by unique factorization it follows that $u_1$ is a multiple of $w_1$ and we are thus in the previous case again.
\end{enumerate}
It therefore follows that the matrix
$$M= \begin{bmatrix} u_1& \ldots & u_j & F' \\
w_1 & \ldots & w_j & G'
\end{bmatrix}
$$
is full rank over $\C(V)$. Thus the determinant of\footnote{$M^\dagger$ is the conjugate transpose of $M$.} $M\cdot M^\dagger$ is a nonzero polynomial over $V$. The Schwartz-Zippel-DeMillo-Lipton lemma now implies that sending each basis element of $V$ to a random multiple of $z$ will make the determinant nonzero with probability $1$. This also means that $F$ and $G$ remain linear independent after such mapping.
\end{proof}


\subsection{An important special case}

Before proving \autoref{thm:main-sg-intro}, we prove a special case where there is a set of quadratics $\cI$, and a vector space of linear forms $V$, and each  quadratic in our  set is a linear combination of quadratics from $\cI$ and a quadratics defined over $V$, and, all nonzero polynomials in the span of $\cI$ remain of rank at least $2$ even when we set the functions in $V$ to zero. We show that in this case the linear forms in $\cL$ satisfy the Sylvester-Gallai condition among themselves. 

\begin{claim}\label{cla:linear-remainder}
Let $\cQ\cup \cL$ satisfy the assumption of \autoref{thm:main-sg-intro} where 
\begin{enumerate}
\item $\cQ$ consists of irreducible quadratics.
\item There is a set of polynomials  $\cI$ and an $O(1)$-dimensional space $V$ such that every polynomial in $\cQ$ is in the linear span of $\cI$ and quadratics over $V$. Furthermore, no nonzero linear combination of the polynomials in $\cI$ can be expressed as $xa+yb+F(V)$ where $F$ is any quadratic over $V$ and $x,a,b,y$ are any four linear forms.
\item $\cL$ is a set of squares of linear functions. 
\end{enumerate}
Then, the dimension of the space spanned by the functions whose squares are in $\cL$ has dimension $O(1)$.
\end{claim}

\begin{proof}
Denote $\cL' = \cL\setminus V$. We  shall prove that  the linear functions in $\cL'$ satisfy the Sylvester-Gallai condition and hence their span has dimension  $O(1)$ as claimed.

Let $x,y\in \cL'$. Let $Q$ be such that $Q\in \sqrt{(x,y)}$. Thus, there exist linear functions $a,b$ so that $Q=xa+yb$. We next consider two cases for $Q$. 

If $Q\in \cQ$ then $Q= Q' + G(V)$, where $Q'$ is a linear combination of the polynomials in $\cI$. In particular, 
$Q' = xa+yb - G(V)$. This implies that $Q'=0$ as otherwise we get a contradiction to  the assumptions on $\cI$ and $V$. Hence, $x a + y b=Q=G(V)$. As $Q$ is irreducible it must hold that $x,y\in V$ (by \autoref{cla:irr-quad-span}). This is in contradiction to the definition of $\cL'$.

The remaining case is when $Q\in \cL$. Thus, $Q=\ell^2$ for some linear  $\ell$, and it follows that $\ell \in \spn\{x,y\}$. Note however that we may have $\ell \in V$. To overcome this we  apply a random projection to the linear functions in $V$ so that they are all equal to some multiple of a new variable $z$. As before it is not hard to see that even after this projection any two linear functions from $\cL'$ are projected to  linearly independent linear functions. Hence, in the case above, there is a third linear function in $\cL' \cup \{z\}$ that is spanned by $x,y$.
It follows that $\cL'\cup\{z\}$ satisfy the conditions of \autoref{cor:bdwy} (with, say, $\delta=1/2$) and hence $\dim(\cL')=O(1)$ as claimed.
\end{proof}

\subsection{The proof}

We are now ready to prove  \autoref{thm:main-sg-intro}. The proof follows the outline sketched in \autoref{sec:proof-idea} and it relies on the claims proved in \autoref{sec:claims-sg} and on \autoref{cor:bdwy}.

\begin{proof}[Proof of \autoref{thm:main-sg-intro}]
Partition the polynomials to two sets. Let $\cL$ be the set of all squares and let $\cQ$ be the subset of irreducible quadratics. Denote $|\cQ|=m_1$.

We next focus on polynomials in $\cQ$. We prove that they are contained in an $O(1)$-dimensional space of a special form.

Call a polynomial $Q\in \cQ$ bad if there are less than, say, $m_1/100$ pairs $(Q_1,Q_2)\in \cQ\times \cQ$ so that $Q_2 \in \sqrt{(Q,Q_1)}$ and $Q,Q_1$ satisfy \autoref{case:span} of \autoref{thm:structure} (i.e. $Q_2$ is in their linear span). If $Q\in\cQ$ is not bad then we call it a good polynomial. We handle two cases according to whether there is at most one bad polynomial or more than that.

\begin{enumerate}
\item {\bf There is at most one bad polynomial:} \hfill

In this case, from \autoref{cor:bdwy} we get that the linear span of the polynomials in $\cQ$ has dimension $O(1)$.

Assume $Q_1,\ldots,Q_k$ for some $k=O(1)$ span $\cQ$. We now repeat the following process. We start with $\cI=\{Q_1,\ldots,Q_k\}$ and $V=\emptyset$. If there is some nontrivial linear combination of the polynomials in $\cI$ that is equal to a quadratic of the form $a_1b_1+ a_2b_2$, where $a_i,b_i$ are linear functions then we add $a_1,a_2,b_1,b_2$ to $V$ and remove one of the polynomials that participated in the linear combination from $\cI$. We continue doing so according to the following rule. If there exists a linear combination  of the polynomials in $\cI$ that is equal to a polynomial of the form $F(V) + a b + a' b'$, where $F(V)$ is a quadratic polynomial over linear functions in $V$, then we add $a,b,a',b'$ to $V$ and  remove some polynomial participating in the linear combination from $\cI$.
We do so until no such linear combination exists or until $\cI$ is empty. At the end $|V|\leq 4k = O(1)$. Abusing notation we now think of $V$ as the space spanned by the linear functions in it. Clearly $\dim(V)\leq 4k = O(1)$.

The argument above implies that the conditions of \autoref{cla:linear-remainder} are satisfied by our $\cI$, $V$, $\cQ$ and $\cL$. We thus obtain that  $\dim(\cL)=O(1)$. Combined with the fact that $|\cI|=O(1)$ this completes the proof for the case when there is at most one bad polynomial. We handle the other case next.

\item {\bf There are at least two bad polynomials:}

\begin{claim}[At least two bad polynomials]\label{cla:2-bad}
If $\cQ$ contains at least two bad polynomials, $Q_1$ and $Q_2$, then there is a space $V$ of linear functions of dimension $O(1)$ so that every polynomial in $\cQ$ is a linear combination of $Q_1$ and a quadratic over $V$.
\end{claim}

\begin{proof}
Notice that for  $Q_1$ there are $0.99m_1$ polynomials in $\cQ$ that even together with $Q_1$ do not span any other polynomial in $\cQ$. The same holds for $Q_2$. Consider a polynomial $Q_j$ so that $Q_1$ and $Q_j$ do not span any other polynomial in $\cQ$. We conclude that  $Q_1$ and $Q_j$ satisfy Case~\ref{case:rk1} or Case~\ref{case:2} of \autoref{thm:structure}. Indeed, if $Q_1$ and $Q_j$ satisfy  Case~\ref{case:span} of \autoref{thm:structure} then they span some polynomial in $\cL$ and in particular they span a square, but this means that they also satisfy Case~\ref{case:rk1}  of \autoref{thm:structure}.

From the discussion above it follows that there are at least $0.98m_1$ polynomials in $\cQ$ satisfying Case~\ref{case:rk1} or Case~\ref{case:2} of the theorem with $Q_1$ and $Q_2$. Let $\cF$ be the set of these polynomials. Partition $\cF$ to three sets $\cI,\cJ,\cK$ so that those polynomials in $\cI$ satisfy Case~\ref{case:2} of \autoref{thm:structure} with   $Q_1$, those in $\cJ$ satisfy  Case~\ref{case:2} of \autoref{thm:structure}   with $Q_2$ and those in $\cK$ satisfy  Case~\ref{case:rk1} of \autoref{thm:structure} with both $Q_1$ and $Q_2$. From \autoref{cor:2-2} and \autoref{cla:3} we conclude that there is a an $O(1)$-dimensional space $V'$ of linear functions such that all those $0.98m_1$ polynomials  are in the linear span of quadratics over $V'$ and $Q_1$.

To simplify things further, if it is the case that $Q_1 = F(V') + aa'+bb' $, i.e. that $Q_1$ can be written as a quadratic over $V'$ plus two products of linear forms, then we add $a,a',b,b'$ to $V'$ and we do not consider $Q_1$ any more.\footnote{This step is not crucial at this point, it just makes some later argument a bit simpler.}

We now consider the remaining $0.02m_1$ polynomials in $\cQ$. In fact, consider those polynomials that cannot be spanned by quadratics over $V'$ and $Q_1$ and call this set $\cF^c$ (abusing notation). 

\begin{claim}\label{cla:Fc}
For each $Q\in \cF^c$ there are at least $0.96m_1$ polynomials in $\cF$ that satisfy either Case~\ref{case:rk1} or Case~\ref{case:2} of \autoref{thm:structure} with $Q$.
\end{claim}

\begin{proof}
If $Q$ and $F\in \cF$ span a polynomial in $\cL$ then we say that $Q$ satisfies Case~\ref{case:rk1} with $F$.  Thus, if $Q$ and $F\in \cF$ satisfy \autoref{case:span}  of \autoref{thm:structure} then the third polynomial is not in $\cF$ (as by switching sides we will get that $Q$ is also in $\cF$). Hence, this polynomial must be in $\cF^c$. Assume that $Q'$ is this polynomial. Notice that there is no other $F'\in\cF$ that together with $Q$ spans $Q'$ as in such a case $Q$ would be in $\cF$. Indeed, let $\alpha_1Q+F=Q'$ and $\alpha_2Q+F'=Q'$. Since $F$ and $F'$ are linearly independent we get that $0\neq (\alpha_1-\alpha_2)Q=F'-F$ in contradiction to the assumption that $Q$ is in $\cF^c$. Thus, $Q$ can satisfy \autoref{case:span}  of \autoref{thm:structure} with at most $|\cF^c|\leq 0.02m_1$ polynomials. 
It follows that there are at least $0.96m_1$ polynomials in $\cF$ that satisfy either Case~\ref{case:rk1} or Case~\ref{case:2}   of \autoref{thm:structure} with $Q$.
\end{proof}

We next show that all polynomials $Q\in \cF^c$ satisfy Case~\ref{case:2}  of \autoref{thm:structure} with some polynomial in $\cF$. Indeed, if this is not the case then there must be a polynomial $Q$  that satisfy Case~\ref{case:rk1}  of \autoref{thm:structure} with all polynomial in $\cF$. 
Let $F_1,F_2\in \cF$. Then, after rescaling, there are $a_1,a_2$ so that $Q+a_1^2 = F_1$ and $Q+a_2^2 = F_2$. Hence, $a_1^2-a_2^2 = F_2-F_1$. As $F_2-F_1$ is a linear combination of $Q_1$ and quadratics over $V'$, it must be the case that $F_2-F_1$ are defined over $V'$ alone as otherwise we would have replaces $Q_1$ with two linear functions as described above. Thus, $a_1^2-a_2^2 =F(V')$ and it follows that $a_1,a_2\in V'$ and hence $Q\in \cF$ in contradiction.

We now bound the dimension of $\cF^c$.
By an argument similar to the proof of \autoref{cla:3} it follows that there is an $O(1)$-dimensional space of linear functions, $V''$ such that all polynomials in $\cF^c$ are quadratics over $V''$: We send $V'$ to a random multiple of a new variable $z$. This makes all polynomials in $\cF^c$ to be of the form $zb_i$ and as before the linear functions $\{b_i\}_i\cup \{z\}$ satisfy the usual Sylvester-Gallai condition and we conclude using \autoref{cor:bdwy} (as in the proof of \autoref{cla:3} we repeat this twice for two independent mappings etc.). Set $V$ be the span of $V'' \cup V'$. This completes the proof of \autoref{cla:2-bad}
\end{proof}

It remains to bound the dimension of $\cL$. This however, follows immediately from \autoref{cla:linear-remainder}.

This concludes the proof of the case of two bad polynomials and with it the proof of \autoref{thm:main-sg-intro}.
\end{enumerate}
\end{proof}

\section{Edelstein-Kelly theorem for quadratic polynomials}\label{sec:quad-EK}

In this section we prove \autoref{thm:main-ek-intro}. We  repeat its statement for convenience.




\begin{theorem*}[\autoref{thm:main-ek-intro}]
Let $\cT_1,\cT_2$ and $\cT_3$ be finite sets of homogeneous quadratic polynomials over $\C$ satisfying the following properties:
\begin{itemize}
\item Each $Q\in\cup_i\cT_i$ is either irreducible or a square of a linear function.
\item No two polynomials are multiples of each other (i.e., every pair is linearly independent).
\item For every two polynomials $Q_1$ and $Q_2$ from distinct sets there is a polynomial $Q_3$ in the third set such that $Q_3\in\sqrt{(Q_1,Q_2)}$. 
\end{itemize}
Then the linear span of the polynomials in $\cup_i\cT_i$'s has dimension $O(1)$.
\end{theorem*}

\begin{remark}
As before, the requirement that the polynomials are homogeneous is without lost of generality as homogenization does not affect the property $Q_k\in\sqrt{(Q_i,Q_j)}$. 
\end{remark}

The proof follows a similar outline to the proof of \autoref{thm:main-sg-intro}.

\begin{proof}[Proof of \autoref{thm:main-ek-intro}]

Partition the polynomials in each $\cT_i$ to two sets. Let $\cL_i$ be the set of all squares and $\cQ_i$ be the rest. Denote $|\cQ_i|=m_i$.

\sloppy
Call a polynomial $Q\in \cQ_1$ bad for $\cQ_2$ if there are less than $m_2/100$ polynomials $Q_2\in \cQ_2$ so that $\spn\{Q,Q_2\}$ contains a polynomial from $\cQ_3$, i.e., $Q$ and $Q_2$ satisfy Case~\ref{case:span} of \autoref{thm:structure} (but not Case~\ref{case:rk1}). We say that $Q\in \cQ_1$ is bad for $\cQ_3$ if the equivalent condition is satisfied.
We say $Q\in \cQ_1$ is bad if it is bad for both $\cQ_2$ and $\cQ_3$.
We call the polynomials in $\cQ_2,\cQ_3$ bad and good in a similar way.


We handle two cases according to whether there is at most one bad polynomial for each $\cQ_i$ or not.

\begin{enumerate}
\item {\bf There is at most one bad polynomial for each $\cQ_j$:} \hfill

In this case, in a similar fashion to the first case of \autoref{thm:main-sg-intro}, we get from \autoref{cor:EK-robust} that the linear span of the polynomials in $\cQ\eqdef \cQ_1\cup\cQ_2\cup\cQ_3$ has dimension $O(1)$. 

As in the proof of \autoref{thm:main-sg-intro} we next extend the bound to also include the linear functions in $\cup_i \cT_i$.
Assume $Q_1,\ldots,Q_k$ for some $k=O(1)$ span $\cQ$. We now repeat the following process. We start with $\cI=\{Q_1,\ldots,Q_k\}$ and $V=\emptyset$. If there is some nontrivial linear combination of the polynomials in $I$ that is equal to a quadratic of the form $F(V)+ a_1b_1+ a_2b_2$, where $a_i,b_i$ are linear functions then we add $a_1,a_2,b_1,b_2$ to $V$ and remove one of the polynomials that participated in the linear combination from $\cI$. We continue doing so 
until no such linear combination exists or until $\cI$ is empty. 
At the end $|V|\leq 4k = O(1)$. As before we abuse notation and think of $V$ as the linear space spanned by the linear functions in it.

It remains to bound the dimension of $\cL\eqdef \cL_1\cup \cL_2\cup \cL_3$. We do so in a similar fashion to the proof of \autoref{cla:linear-remainder}. Denote $\cL' = \cL\setminus V$.

First, we apply a random projection to the linear functions in $V$ so that they are all equal to some multiple of $z$. We next show that the set $\cL' \cup \{z\}$ satisfies the Sylvester-Gallai condition and hence its dimension is $O(1)$ as needed (we abuse notation and denote with $\cL'$ the projection of $\cL'$, which, as before, still consists of pairwise independent linear functions).

Let $x,y\in \cL'$ come from two different $\cL_i$. Let $Q$ be such that $Q\in \sqrt{(x,y)}$. If $Q\in \cQ$ then $Q= Q' + G(z)$, where $Q'$ is a linear combination of the polynomials in $\cI$. Note however, that by definition of $V$, $Q'$ must be zero as otherwise we would have a linear combination of small rank and then the set $\cI$ would be different. Hence, $Q=G(z)$. It follows that $z\in \spn\{x,y\}$ and so $x,y,z$ are linearly dependent as required. If, on the other hand, $Q\in \cL$ then $Q=\ell^2$ and it follows that $\ell \in \spn\{x,y\}$. In either case, there is a third linear function in $\cL' \cup \{z\}$ that is spanned by $x,y$ as claimed.

Note that if $\cL\subseteq \cL_i$ for some $i$ then we easily conclude this case by picking any $x\in \cL$ and any $Q$ in a different set and as above conclude that $x\in\spn\{z\}$.

\item {\bf There are at least two bad polynomial for some $\cQ_j$:} \hfill \break
To ease notation assume w.l.o.g. that there are at least two bad polynomials for $\cQ_3$.
The next claim gives something similar to the first part in the proof of \autoref{cla:2-bad}.

\begin{claim}\label{cla:2-bad-i}
Assume $Q_1,Q_2\in\cQ_1\cup\cQ_2$ are bad for $\cQ_3$, then there is a space $V$ of linear functions of dimension $O(1)$ so that at least $0.98m_3$ of the  polynomials in $\cQ_3$ are in the linear span of $Q_1$ and quadratic polynomials over $V$.
\end{claim}

\begin{proof}
Notice that for  $Q_1$ there are $0.99 m_3$ polynomials in $Q'\in \cQ_3$ that even together with $Q_1$ do not span any other polynomial in $\cQ_2$. 
The same holds for $Q_2$. Consider a polynomial $Q'\in\cQ_3$ so that $Q_1$ and $Q'$ do not span any other polynomial in $\cQ_2$. We conclude that  $Q_1,Q'$ satisfy Case~\ref{case:rk1} or Case~\ref{case:2} of \autoref{thm:structure}. Indeed, if $Q_1$ and $Q'$ satisfy  \autoref{case:span} of \autoref{thm:structure} then they span some polynomial in $\cL_2$ and in particular they span a square of a linear function, but this means that they also satisfy Case~\ref{case:rk1}  of \autoref{thm:structure}. \\


From the discussion above it follows that there are at least $0.98 m_3$  polynomials in $\cQ_3$ satisfying Case~\ref{case:rk1} or Case~\ref{case:2} of the theorem with $Q_1$ and $Q_2$. Let $\cF_3$ be the set of these polynomials in $\cQ_3$. We partition $\cF_3$ to three sets $\cI_3,\cJ_3,\cK_3$ so that those polynomials in $\cI_3$ satisfy Case~\ref{case:2} of \autoref{thm:structure}  with $Q_1$, those in $\cJ_3$ satisfy  Case~\ref{case:2} of \autoref{thm:structure}   with $Q_2$ and those in $\cK_3$ satisfy  Case~\ref{case:rk1} of \autoref{thm:structure} with both $Q_1$ and $Q_2$. As before we would like to apply \autoref{cor:2-2} and \autoref{cla:3} to conclude that  there is a an $O(1)$-dimensional space $V'$ of linear functions such that all those $0.98 m_3 $ polynomials  of $\cF_3$ are in the linear span of quadratics over $V'$ and $Q_1$. The only problem is that the proof of \autoref{cla:3} should be tailored to the colored case, which is what we do next (indeed, \autoref{cla:2-2} can be applied without any changes and therefore also \autoref{cor:2-2}). 

Note that if $Q\in\cQ_1$ satisfies Case~\ref{case:2}  of \autoref{thm:structure} with some polynomial in $\cQ_3$ then it also satisfies the same case with a polynomial in $\cQ_2$.

\begin{claim}\label{cla:3-i}
Let $\cI_2\subseteq \cQ_2$ and $\cI_3\subseteq \cQ_3$ be irreducible quadratics  that satisfy Case~\ref{case:2} of \autoref{thm:structure} with an irreducible $Q\in\cQ_1$. Then, there exists an $O(1)$-dimensional space $V$ such that all polynomials in $\cI_2 \cup \cI_3$ are quadratic polynomials in the linear functions in $V$.
\end{claim}

We postpone the proof of the claim to \autoref{sec:missing} and continue with the proof of  \autoref{cla:2-bad-i}.
By applying \autoref{cla:3-i} first to $\cI_3$ and then to $\cJ_3$ we conclude that $\cI_3\cup \cJ_3$ are quadratics over a set of $O(1)$ linear functions $V$. \autoref{cor:2-2} implies that every quadratic in $\cK_3$ is in the linear span of $Q_1$ and quadratics over an $O(1)$-sized set $V'$. combining $V$ and $V'$ the claim follows. This completes the proof of  \autoref{cla:2-bad-i}.
\end{proof}

Let $V$ be the $O(1)$-dimensional space and $\cF_3\subseteq \cQ_3$ the set of polynomials guaranteed by \autoref{cla:2-bad-i}.

To continue we again have to consider two cases. The first is when there are two polynomials that are bad for $\cQ_1$ or for $\cQ_2$ (so far we assumed there are at least two bad polynomials for $\cQ_3$). The second case is when at most one polynomial is bad for $\cQ_1$ and at most one polynomial is bad for $\cQ_2$.

\begin{enumerate}
\item {\bf There are two bad polynomials for some $\cQ_i$, $i\in[2]$:} \hfill

Assume w.l.o.g. that $i=2$. 
As before \autoref{cla:2-bad-i} implies that there is a polynomial $Q_2$ and an $O(1)$-dimensional space $U$ such that $0.98m_2$ of the polynomials in $\cQ_2$ are in the linear span of $Q_2$ and quadratics over $U$. Call those polynomials $\cF_2$. 
Let $W = U+V$ be an $O(1)$-dimensional space containing both $U$ and $V$.

We now check whether there is any nontrivial linear combination of $Q_1$ and $Q_2$ that is of the form $a\cdot b + a'\cdot b' + F(W)$. If such a combination exists then we add $a,a',b,b'$ to $W$ (and abusing notation call the new sets $W$ as well) and replace one polynomial that appeared in this combination with the other. I.e. if $Q_2$ appeared in such a combination then we think of the space that is spanned by $Q_1$ and $W$ rather than by $Q_2$ and $W$. We continue to do so once again if necessary.

Assume further, w.l.o.g., that $|\cQ_2| \geq |\cQ_3|$.  Partition the set $\cQ_1$ to three sets $\cI,\cJ,\cK$ so that:
\begin{flalign}
\text{Each } & \text{$Q\in \cI$ satisfies Case~\ref{case:2} of \autoref{thm:structure} with at least one polynomial in } \cF_2. \label{partition}\\
\text{Each } & \text{$Q\in \cJ$ satisfies Case~\ref{case:rk1} of \autoref{thm:structure} with at least two polynomials in } \cF_2.\nonumber\\
\text{Each } & \text{$Q\in \cK$ satisfies Case~\ref{case:span} of \autoref{thm:structure} with all except possibly  one polynomial in $\cF_2$.}\nonumber
\end{flalign}

\begin{claim}\label{cla:Q1}
With the notation above we prove the following claims.
\begin{enumerate}[label={(\roman*)},itemindent=1em]
\item The linear span of all polynomials in $\cI$ has dimension $O(1)$.\label{cla:I}
\item All polynomials in $\cJ$ are polynomials over $W$.\label{cla:J}
\item All polynomials in $\cK$ are in the linear span of $Q_1,Q_2$ and quadratics over $W$.\label{cla:K}
\end{enumerate}
\end{claim}

\begin{proof}
The proof of \autoref{cla:I} follows exactly as in \autoref{cla:3-i}.

To show \autoref{cla:J} we proceed as in the discussion following the proof of \autoref{cla:Fc}.
Consider a polynomial $Q\in \cJ$. Let $F_1,F_2\in \cQ_2$ satisfy  Case~\ref{case:rk1} of \autoref{thm:structure} with $Q$. Then, after rescaling, there are $a_1,a_2$ so that $Q+a_1^2 = F_1$ and $Q+a_2^2 = F_2$. Hence, $a_1^2-a_2^2 = F_2-F_1$. As $F_2-F_1$ is a linear combination of $Q$ and quadratics over $W$, it must be the case that $F_2-F_1$ are defined over $W$ alone as otherwise we would have replaced $Q_2$ with two linear functions as described above. Thus, $a_1^2-a_2^2 =F(W)$ and it follows that $a_1,a_2\in W$ and hence $Q$ is a polynomial over $W$.

Finally, to prove \autoref{cla:K} we note that for every $Q\in\cK$ there are at least $0.98m_2-1$ polynomials $Q_2\in \cF_2$  so that for each of them there is $Q_3\in\cQ_3\cap\spn\{Q,Q_2\}$. If there exists such a combination where $Q_3\in \cF_3$ then it follows that $Q$ is a linear combination of $Q_1,Q_2$ and quadratics over $W$ (as all polynomials in $\cF_2$ and $\cF_3$ are). If we always get $Q_3 \not\in\cF_3$ then as $|\cQ_3\setminus \cF_3|\leq 0.02 m_3 \leq 0.02m_2 < (1/2)\cdot |\cF_2|$ there exist $Q_2,Q'_2 \in \cF_2$ and $Q_3\in\cQ_3$ so that $Q_3\in \spn\{Q,Q_2\},\spn\{Q,Q'_2\}$. As every two polynomials in our set are linearly independent this implies that $Q\in \spn\{Q_2,Q'_2\}$, and in particular it is in the span of $Q_2$ and quadratics over $W$, as claimed. 
\end{proof}

A similar argument will now show that $\cQ_2$ and $\cQ_3$ are also contained in an $O(1)$-dimensional space. We thus showed that there is an $O(1)$-dimensional space containing all polynomials in $\cQ_1\cup\cQ_2\cup\cQ_3$. It remains to bound the dimension of the linear functions in $\cL_1\cup\cL_2\cup\cL_3$. This can be done at exactly the same way as before.  
This concludes the proof of \autoref{thm:main-ek-intro} in this case.\\

\item {\bf At most one polynomial is bad for $\cQ_1$ and at most one polynomial is bad for $\cQ_2$}\hfill

In this case we reduce to the extended robust Edelstein-Kelly theorem (\autoref{thm:EK-robust-span}).

For each $i\in[2]$ partition $\cQ_i$ to $\cI_i,\cJ_i$ and $\cK_i$ as in \autoref{partition} except that we now consider $\cF_3$ instead of $\cF_2$ when partitioning. It follows, exactly as in the proof of \autoref{cla:Q1}, that there is an $O(1)$-dimensional space $U$ that all polynomials in $\cI_1\cup\cJ_1\cup\cI_2\cup\cJ_2$ are in the linear span of $Q_1$ and quadratics over $U$.

Let $W$ be the space spanned by  $Q_1$ and quadratics over $U$. Clearly $\dim(W)=O(1)$.

For $i\in[2]$ let $\cK'_i\subset\cK_i$ be those polynomials in $\cK_i$ that are not in $W$. Similarly, define $K'_3\subset\cQ_3$. Let $W_i  = W \cap \cQ_i$, for $i\in [3]$.

We now observe that the sets $\cQ_1=\cK'_1\cup W_1, \cQ_2=\cK'_2\cup W_2, \cQ_3=\cK'_3\cup W_3$ satisfy the conditions in the statement of \autoref{thm:EK-robust-span} (where the $\cK_i$ in the statement of the theorem is our $\cK'_i$), with parameters $r=O(1)$, $c=2$ and $\delta = 1/100$, when we identify our quadratic polynomials with their vectors of coefficients.

Indeed, as we are in the case where there is at most one bad polynomial for $\cQ_1$ and at most one bad polynomial for $\cQ_2$ we see that there are at most $2$ ``exceptional'' vectors defined that way. Furthermore, from the definition of $\cK'_1,\cK'_2$ (\autoref{partition}) no point in them is ``exceptional'' when considering $\cQ_3$.

Thus, \autoref{thm:EK-robust-span} guarantees the existence of a space $Y$ of dimension $O_c(r+1/\delta^3)=O(1)$ that spans all vectors in the set $\cQ_1\cup\cQ_2\cup\cQ_3$. 
We are almost done - we still have to deal with the linear function in $\cL_1\cup\cL_2\cup\cL_3$. This however is done exactly as before.
\end{enumerate}
\end{enumerate}
This completes the proof of \autoref{thm:main-ek-intro} (modulo the proof of \autoref{cla:3-i} that we give next).
\end{proof}

\subsection{Missing proof}\label{sec:missing}

In this section we give the proof of \autoref{cla:3-i}. For convenience we repeat the statement of the claim.

\begin{claim*}[\autoref{cla:3-i}]
Let $\cI_2\subseteq \cQ_2$ and $\cI_3\subseteq \cQ_3$ be irreducible quadratics  that satisfy Case~\ref{case:2} of \autoref{thm:structure} with an irreducible $Q\in\cQ_1$. Then, there exists an $O(1)$-dimensional space $V$ such that all polynomials in $\cI_2 \cup \cI_3$ are quadratic polynomials in the linear functions in $V$.
\end{claim*}

\begin{proof}[Proof of \autoref{cla:3-i}]
Let $\cI_2 = \{F_i\}_i$ and $\cI_3=\{G_i\}_i$. 
As before we take $V'$ to be the space spanned by the linear functions in a minimal representation of $Q$. Clearly $\dim(V')\leq 4$. Let $z$ be a new variable. Set each basis element of $V'$ to a random multiple of $z$ (as before, we pick the multiples independently, uniformly at random from $[0,1]$). Each $F_i,G_i$ now becomes $z\cdot b_i$ for some nonzero $b_i$. Indeed, if we further set $z=0$ then all linear functions in the representation of $Q$ vanish and hence  $F_i$ and $G_i$ also vanish.\footnote{Here too we use the fact that $Q$ is irreducible and hence the two linear functions that make $F_i$ (or $G_i$) vanish appear in $V'$ (\autoref{cla:irr-quad-span}).} Furthermore, for any $i\neq j$, $b_i$ and $b_j$ are linearly independent (as in \autoref{cla:still-indep}), unless they both equal to multiples of $z$. 

Let $\cI_1$ be the set of quadratics in $\cQ_1$ that after making the restriction become quadratics of the form $z\cdot b$. Clearly $Q_1$ is such a polynomial. 

We next show that the linear functions $\{b_i\}_i\cup \{z\}$, where the $b_i$ are the linear functions coming from $\cI_1\cup\cI_2\cup\cI_3$ as described above, satisfy the usual Sylvester-Gallai condition and conclude by \autoref{thm:bdwy}  that their rank is $O(1)$.

\begin{claim}\label{cla:both-sets}
If some polynomial in $\cI_2$ ($\cI_3$) is projected to $z\cdot b$ where $b$ is linearly independent of $z$ then there is some polynomial in $\cI_3$ ($\cI_2$) that is projected to $z\cdot c$ for some $c$ linearly independent of $z$.
\end{claim}

\begin{proof}
Consider any polynomial $Q'\in \cI_2$ that was projected to $z\cdot b$,  where $b$ is linearly independent of $z$, and let $Q''\in\cI_3$ be in $\sqrt{(Q_1,Q')}$. Assume for a  contradiction that $Q''$ was projected to $z^2$. \autoref{cla:still-indep} implies that if this is the case then all linear functions in a minimal representation of $Q''$ belong to $V'$. 

We thus have that $Q''=Q''(V')$. We can also assume w.l.o.g. that $Q'=Q'(V',b)$ (by completing $V'\cup\{b\}$ to a basis for the entire space of linear functions and projecting the other basis elements to random multiples of $b$). We next show that $Q''\in \sqrt{Q_1}$, which implies $Q''$ is a multiple of $Q_1$ in contradiction. 

We again resort to \autoref{thm:structure}. It is clear that $Q''\not\in\spn\{Q_1,Q'\}$. So we are left with the two other cases.

\begin{enumerate}
\item $Q_1$ and $Q'$ span a square of a linear function: It is not hard to see that in this case we must have (after rescaling) that $Q' = b^2 + \ell(V')\cdot b +A'(V')$. Consider any assignment to $V'$ that makes $Q_1$ vanish. Clearly there is a value to $b$ that also makes $Q'$ vanish for that assignment. Thus $Q''$ also vanishes. Therefore, any assignment that makes $Q_1=0$ also makes $Q''=0$ which is what we wanted to prove.

\item There are two linear functions $v_1,v_2\in V'$ so that $Q_1,Q',Q''\in \sqrt{(v_1,v_2)}$: Denote $Q_1 = v_1\cdot u_1 + v_2 \cdot u_2$, $Q' + v_1 \cdot b_1 + v_2 \cdot b_2$ and $Q'' =v_1\cdot w_1 + v_2 \cdot w_2$, where $w_i,u_i\in V'$. Project $v_1$ and $v_2$ to random multiples of a new variable $y$. Then, our new polynomials are now $Q_1=y\cdot u$, $Q' = y\cdot b'$ and $Q''=y\cdot w$, where $u,w\in V'$ (where we abuse notation and refer to the projection of $V'$ also as $V'$) and, with probability $1$, $b'\not \in V'$.
Consider the assignment $u=b'=0$. It follows that we also get $y\cdot w=0$. However, as $y,w,u\in V'$ and $b'\not\in V'$ it must be the case that $y\cdot w=0$ modulo $u$. Thus, after this projection we get that $Q''\in\sqrt{(Q_1)}$. This implies however that $Q''$ is a multiple of $Q_1$ as it cannot be the case that $Q_1$ was projected to a square (as this would imply that it was only a function of $v_1$ and $v_2$ and hence a reducible polynomial). \autoref{cla:still-indep} implies that this was also the case before the projection, in contradiction.
\end{enumerate}
\end{proof}

We continue with the proof of \autoref{cla:3-i}. \autoref{cla:both-sets} establishes that either all polynomials in $\cI_2\cup \cI_3$ were projected to $z^2$ or that both $\cI_2$ and $\cI_3$ contain polynomials that were projected to quadratics of the form $z\cdot b$ where $b$ is linearly independent of $z$. 

We are now ready to show that the linear functions $\{b_i\}_i\cup \{z\}$, where $b_i$ are the linear functions in $\cI_1\cup\cI_2\cup\cI_3$, satisfy the usual Sylvester-Gallai condition.

Consider any two quadratics $A_2=z\cdot b_2\in\cI_2,A_3=z\cdot b_3\in\cI_3$ so that  neither $b_2$ nor $b_3$ is a multiple of $z$. If $b_2$ and $b_3$ span $z$ then we are done. So assume that $z\not\in\spn\{b_2,b_3\}$. Let $A_1$ vanish when $A_2,A_3$ vanish. Then clearly $z$ divides $A_1$. Thus $A_1=z\cdot b_1$ is in $\cI_1$ and so $b_1$ is in our set. Further, when we set $b_2=b_3=0$ both $A_2,A_3$ vanish and hence also $A_1$ vanishes. Since $z\not \in \spn\{b_1,b_2\}$ this implies that $b_1 \in  \spn\{b_2,b_3\}$ and so in this case $b_2$ and $b_3$ span a third linear function in our set. Note also that by \autoref{cla:still-indep} $b_1$ is not a multiple of $b_2$ nor of $b_3$ as this would imply that $A_1$ and $A_2$ (or $A_3$) are linearly dependent in contradiction to our assumption. 

This argument shows that whenever $b_2$ and $b_3$ are not a multiple of $z$ (and they come from different sets), the set $\{b_i\}_i\cup \{z\}$ contains a nontrivial linear combination of them. In a similar fashion to  \autoref{cor:bdwy} and \autoref{thm:EK-robust-span} we get that the dimension of all those linear functions is $O(1)$.\footnote{We note that we cannot apply \autoref{thm:EK-robust-span} as is as it may be the case that $z$ appears in all three sets. However, it is not hard to see that a small modification of it will capture this case as well.}

As in the proof of \autoref{cla:3} we repeat this argument again for a different random mapping to multiples of $z$ and conclude in the same way that every polynomial in $\cI_2\cup\cI_3$ is a polynomial over some $O(1)$-dimensional space $V$.

This completes the proof of \autoref{cla:3-i}.
\end{proof}



\section{Conclusions and future research}\label{sec:discussion}

In this work we proved analogs of theorems of Sylvester-Gallai and Edelstein-Kelly for quadratic polynomials. These results directly relate to the problem of obtaining deterministic algorithms for testing identities of  $\Sigma^{[3]}\Pi^{[d]}\Sigma\Pi^{[2]}$ circuits. As mentioned in \autoref{sec:intro} in order to obtain PIT algorithms we need even stronger extensions of these results - something in the line of \autoref{con:gupta-general} that was proposed by Gupta \cite{Gupta14}. 


It is quite likely that Theorems~\ref{thm:main-sg-intro} and \ref{thm:main-ek-intro} could be extended to obtain a positive answer to \autoref{con:gupta-general}  for $r=2$ and $k=3$. Indeed, there is an analog of \autoref{thm:structure-intro} that suits the condition of the conjecture (for $r=2$ and $k=3$). Peleg \cite{Peleg19} used this extension of \autoref{thm:structure-intro} to generalize 
\autoref{thm:main-sg-intro} to the case where for every $Q_i$ and $Q_j$ it holds that whenever they vanish the product of the other $Q_k$'s vanishes as well. This is a significant step towards resolving \autoref{con:gupta-general}  (for $r=2$ and $k=3$). 

However, extending our approach to the case of more than $3$ multiplication gates (or more than $3$ sets as in 
\autoref{thm:main-ek-intro}) seems more challenging. Indeed, the structure theorem gets more complicated in the sense that there are many more cases to consider and it seems unlikely that a similar approach will work for ``higher values of $3$''. Similarly, while proving a structural theorem for degree $3$ polynomials is possible, it seems that extending the exact same approach to significantly higher degrees may be less easy. Thus, we believe that a different proof approach may be needed in order to obtain PIT algorithms for $\Sigma^{[O(1)]}\Pi^{[d]}\Sigma\Pi^{[O(1)]}$ circuits.

Another interesting question is, stated vaguely, understanding the conditions under which we get a Sylvester-Gallai kind of behavior. By now many variants of the theorem are known: The original Sylvester-Gallai theorem, the colored version of it (Edelstein-Kelly theorem), robust versions of it (by \cite{barak2013fractional,DSW12}), extensions to subspaces \cite{DBLP:journals/dcg/DvirH16}, $k$-wise dependencies \cite{Hansen65,barak2013fractional}, our results for quadratic polynomials and more. It is an intriguing question whether there is a common generalization of all these cases or some framework that contain all these different results.





\section*{Acknowledgments} 
I would like to thank Shir Peleg for helpful discussions and Ankit Gupta for commenting on an earlier version of the paper. I also thank the anonymous reviewer for their comments.
Part of this work was done while the author was visiting NYU.


\bibliographystyle{amsplain}


\begin{dajauthors}
\begin{authorinfo}[pgom]
  Amir Shpilka\\
  Professor\\
  Tel Aviv University\\
  Tel Aviv, Israel\\
  shpilka\imageat{}tauex\imagedot{}tau\imagedot{}ac\imagedot{}il   \\
  \url{https://www.cs.tau.ac.il/~shpilka}
\end{authorinfo}
\end{dajauthors}

\end{document}